\documentclass[11 pt, reqno]{amsart}

\usepackage{amsfonts}
\usepackage{amssymb, amsmath, amsthm, enumitem,calrsfs,stmaryrd}
\usepackage[colorlinks=true, linkcolor = blue, citecolor= Green]{hyperref}
\usepackage[alphabetic,lite]{amsrefs}
\usepackage{amscd}   
\usepackage{fullpage}
\usepackage[all]{xy} 
\usepackage{faktor}
\usepackage[utf8x]{inputenc}
\usepackage{verbatim}
\usepackage{mathrsfs}
\usepackage{epstopdf}
\usepackage[normalem]{ulem}
\usepackage{setspace}
\DeclareMathAlphabet{\mathpzc}{OT1}{pzc}{m}{it}
\DeclareFontEncoding{OT2}{}{} 
\newcommand{\textcyr}[1]{%
 {\fontencoding{OT2}\fontfamily{wncyr}\fontseries{m}\fontshape{n}\selectfont #1}}
\newcommand{\Sha}{{\mbox{\textcyr{Sh}}}}

\usepackage[usenames,dvipsnames]{color}

\newcommand{\yuan}[1]{{\color{Orange} \sf $\clubsuit\clubsuit\clubsuit$ Yuan: [#1]}}

\makeatletter
\def\@tocline#1#2#3#4#5#6#7{\relax
  \ifnum #1>\c@tocdepth 
  \else
    \par \addpenalty\@secpenalty\addvspace{#2}%
    \begingroup \hyphenpenalty\@M
    \@ifempty{#4}{%
      \@tempdima\csname r@tocindent\number#1\endcsname\relax
    }{%
      \@tempdima#4\relax
    }%
    \parindent\z@ \leftskip#3\relax \advance\leftskip\@tempdima\relax
    \rightskip\@pnumwidth plus4em \parfillskip-\@pnumwidth
    #5\leavevmode\hskip-\@tempdima
      \ifcase #1
       \or\or \hskip 1em \or \hskip 2em \else \hskip 3em \fi%
      #6\nobreak\relax
    \hfill\hbox to\@pnumwidth{\@tocpagenum{#7}}\par
    \nobreak
    \endgroup
  \fi}
\makeatother

\newtheorem{lemma}{Lemma}[section]
\newtheorem{theorem}[lemma]{Theorem}
\newtheorem{proposition}[lemma]{Proposition}

\newtheorem{corollary}[lemma]{Corollary}

\newtheorem{claim*}{Claim}

\newtheorem{definition}[lemma]{Definition}
\newtheorem{example}[lemma]{Example}

\theoremstyle{definition}
\newtheorem{remark}[lemma]{Remark}

\newcommand{\C}{{\mathbb C}}
\newcommand{\F}{{\mathbb F}}
\newcommand{\Q}{{\mathbb Q}}
\newcommand{\R}{{\mathbb R}}
\newcommand{\Z}{{\mathbb Z}}

\newcommand{\calF}{{\mathcal F}}
\newcommand{\calG}{{\mathcal G}}

\newcommand{\calO}{{\mathcal O}}

\newcommand{\calS}{{\mathcal S}}
\newcommand{\calT}{{\mathcal T}}

\newcommand{\frakp}{{\mathfrak p}}
\newcommand{\frakq}{{\mathfrak q}}

\newcommand{\frakP}{{\mathfrak P}}

\usepackage[mathscr]{euscript}

\DeclareMathOperator{\Char}{char}

\DeclareMathOperator{\im}{im}

\DeclareMathOperator{\Hom}{Hom}
\DeclareMathOperator{\Ext}{Ext}

\DeclareMathOperator{\Gal}{Gal}

\DeclareMathOperator{\Cl}{Cl}

\DeclareMathOperator{\B}{{\mbox{\textcyr{B}}}}

\DeclareMathOperator{\gr}{gr}

\DeclareMathOperator{\Rad}{Rad}

\DeclareMathOperator{\ab}{ab}

\DeclareMathOperator{\Nm}{Nm}
\DeclareMathOperator{\Ram}{Ram}
\DeclareMathOperator{\el}{el}
\DeclareRobustCommand\longtwoheadrightarrow
     {\relbar\joinrel\twoheadrightarrow}

\newcommand{\Conv}{\mathop{\scalebox{1.5}{\raisebox{-0.2ex}{$\ast$}}}}
\newcommand{\cHom}{\mathcal{H}\!\!om}

\numberwithin{equation}{section}
\numberwithin{table}{section}

\usepackage{tikz-cd}

\title{On the $p$-rank of class groups of $p$-extensions}
\author{Yuan Liu}
\address{Department of Mathematics\\
University of Illinois Urbana-Champaign \\ 1409 W Green St \\ Urbana, IL 61801 USA}  
\email{yyyliu@illinois.edu}

\begin{document}

	 \begin{abstract}
	 	We prove a local-global principle for the embedding problems of global fields with restricted ramification. By this local-global principle, for a global field $k$, we use only the local information to give a presentation of the maximal pro-$p$ Galois group of $k$ with restricted ramification, when some Galois cohomological conditions are satisfied. For a Galois $p$-extension $K/k$, we use our presentation result for $k$ to study the structure of pro-$p$ Galois groups of $K$. Then for $k=\Q$ and $k=\F_q(t)$ with $p\nmid q$, we give upper and lower bounds for the rank of $p$-torsion group of the class group of $K$, and these bounds depend only on the structure of the Galois group and the inertia subgroups of $K/k$. Finally, we study the $p$-rank of class groups of cyclic $p$-extensions of $\Q$ and the $2$-rank of class groups of multiquadratic extensions of $\Q$, for a fixed ramification type. 
	 
	 \end{abstract}

	\maketitle
	
	\hypersetup{linkcolor=black}
	\tableofcontents
	\hypersetup{linkcolor=blue}

\section{Introduction}

	For a quadratic extension $K/\Q$, by Gauss's genus theory, the 2-torsion subgroup of the class group of $K$, denoted by $\Cl(K)[2]$, is determined by the number of primes ramified in $K/\Q$. Similarly, for a cyclic degree-$p$ extension $K/\Q$, it is well-known that one can use the number of ramified primes to compute the $\Gal(K/\Q)$-coinvariant of $\Cl(K)/p\Cl(K)$, because it corresponds to the genus field of $K$. Gerth (\cite{Gerth84, Gerth86}) conjectured that the distribution of the $p$-primary part of $\Cl(K)$ that is not determined by the genus theory is given by a probability measure similar to the ones in the Cohen--Lenstra heuristics. Recently, Gerth's conjecture is completely proven by Smith, Koymans and Pagano (\cite{Smith, Koymans-Pagano-Gerth, Smith2022}). One of the reasons that Gerth's conjecture is less difficult than the Cohen--Lenstra heuristics is that the $p$-part of $\Cl(K)$ of a cyclic degree-$p$ extension $K/\Q$ can be described using only the local information (more precisely, the Frobenius elements at the ramified primes). In this paper, we will extend this local-global principle to study the maximal pro-$p$ extension with restricted ramification of a Galois $p$-extension of global fields.
	
	Let $k$ be a global field, $S$ and $T$ two finite sets of primes of $k$, and $p$ a prime number not equal to $\Char(k)$. Let $k_S^T$ denote the maximal extension of $k$ that is unramified outside $S$ and splitting completely at primes in $T$, and $G_S^T(k)$ denote the Galois group $\Gal(k_S^T/k)$. 
	Let $G_S^T(k)(p)$ denote the pro-$p$ completion of $G_S^T(k)$, and $k_S^T(p)$ denote maximal pro-$p$ extension of $k$ inside $k_S^T$.
	It is well-known that $G_S^T(k)(p)$ is a finitely generated pro-$p$ group, whose generator rank and relator rank can be computed by formulas given by Shafarevich and Koch \cite{Koch}. Let $d$ be the generator rank of $G_S^T(k)$ and $F_d$ the free pro-$p$ group on $d$ generators, and consider a surjection of pro-$p$ groups
	\begin{equation}\label{eq:pres-intro}
		\varphi: F_d \longtwoheadrightarrow G_S^T(k)(p).
	\end{equation}
	For each prime $\frakp$ of $k$, a decomposition subgroup of $G_S^T(k)(p)$ at $\frakp$ is naturally a quotient of $\calG_{\frakp}(p)$, the pro-$p$ completion of the absolute Galois group of the local completion $k_{\frakp}$ of $k$ at $\frakp$. The restriction of $\varphi$ gives a presentation of a decomposition subgroup at $\frakp$. So a presentation of $\calG_{\frakp}(p)$, which is well-studied, provides some relators for the presentation \eqref{eq:pres-intro}, and we call such type of relators \emph{the local relators at $\frakp$}.
	
	For example, let $k$ be $\Q$. For a nonarchimedean prime $\frakp\in S$ of $\Q$ (so $\frakp$ is a prime number) such that $\frakp \neq p$, if $\frakp$ is ramified in $\Q_S^T(p)/\Q$, then it must be tamely ramified. By the result of Iwasawa \cite{Iwasawa}, the Galois group of the maximal tamely ramified extension of $\Q_\frakp$ is topologically generated by two elements $\tau$ and $\sigma$ with only one relator $\tau^\frakp=\sigma \tau \sigma^{-1}$. So a decomposition subgroup of $G_S^T(\Q)(p)$ at $\frakp$ is generated by two elements $t$ and $s$ such that $t^{\frakp}=s t s^{-1}$, and hence all the elements of $\varphi^{-1}(t^{\frakp}s t^{-1} s^{-1})$ are contained in $\ker \varphi$. For any choice of lifts $x \in \varphi^{-1}(t)$ and $y \in \varphi^{-1}(t)$, we call the element $x^{\frakp} y x^{-1} y^{-1}$ a \emph{local relator} at $\frakp$.
	Similarly, at the archimedean prime of $\Q$, a local relator is $x^2$, where $x$ is a lift of a complex conjugation of $\Q_S^T(p)$, and such local relators are automatically contained in $\ker \varphi$.
	
	In this paper, we first study the question: \emph{for which $k$, $S$ and $T$, is $\ker \varphi$ the closed normal subgroup generated by all the local relators?} We show this question is related to the local-global principle of embedding problems with restricted ramification. The local-global principle of embedding problems is a crucial technique used in the inverse Galois problems for solvable groups \cite{Scholz, Reichardt, Iwasawa-solvable, Sha-190}, and the local-global principle for central embedding problems with no further ramification over $\Q$ and $\F_q(t)$ is proven in the work of Wood, Zureick-Brown and the author on non-abelian Cohen--Lenstra heuristics \cite{LWZB}.
	There exists a finite abelian group $\B_{S\backslash T}^{S \cup T}(k,\F_p)$, which was defined to determine the generator rank and relator rank of $G_S^T(k)(p)$ (see \cite{Koch} for the case when $T=\O$). 
	We prove that the above question has a positive answer when this group $\B_{S\backslash T}^{S \cup T}(k, \F_p)$ is trivial. 
	
	\begin{theorem}[Theorem~\ref{thm:pres-ST}\eqref{item:pres-ST-2} and Remark~\ref{rmk:pres-ST}] \label{thm:main-1}
		Let $k$ be a global field and $p$ a prime number not equal to $\Char(k)$. If $S$ and $T$ are two finite sets of primes of $k$ such that $\B_{S\backslash T}^{S \cup T}(k, \F_p)=0$, then for any surjection $\varphi$ described as \eqref{eq:pres-intro}, $\ker \varphi$ is the closed normal subgroup of $F_d$ generated by all local relators at primes in $S \backslash T$. 
	\end{theorem}
	
	We also show in Theorem~\ref{thm:pres-ST}\eqref{item:pres-ST-3} that, when some additional conditions are satisfied, we can give a generator set for $F_d$ using the preimages of the generators of inertia subgroups at $\frakp$ in some subset of $S\backslash T$, and therefore, we can give a presentation of $G_S^T(k)(p)$ using only the local information.
	The condition in Theorem~\ref{thm:main-1} is satisfied in many situations: for each $k$, there are infinitely many choices of $S$ and $T$ such that $\B_{S\backslash T}^{S \cup T}(k, \F_p)=0$ (see Lemmas~\ref{lem:B=0} and \ref{lem:B=0-empty}). In particular, this together with all the additional conditions in Theorem~\ref{thm:pres-ST} are satisfied in two important cases:
	\begin{enumerate}
		\item\label{item:intro-Q} $k=\Q$, $T=\O$, and $S$ is any finite set of primes that contains the archimedean prime.
		\item\label{item:intro-Fq(t)} $k=\F_q(t)$, $T=\{\infty\}$, and $S$ is any finite set of primes.
	\end{enumerate}
	
	Suppose that $K/k$ is a Galois subextension of $k_S^T(p)/k$. For a subset $S'$ of $S$, we let $K_{S'}^T(p)$ denote the maximal pro-$p$ extension of $K$ that is unramified at all primes not lying above $S'$ and splitting completely at each prime above $T$. Then $K_{S'}^T(p)$ is a subfield of $k_S^T(p)$, and hence we have the following surjections
	\[
		G_S^T(k)(p) \longtwoheadrightarrow \Gal(K_{S'}^T(p)/k) \longtwoheadrightarrow \Gal(K/k).
	\]
	When $G_S^T(k)(p)$ can be presented using only the local information, we ask if $\Gal(K_{S'}^T(p)/k)$ can also be presented using the local information. For a prime $\frakp$ of $k$, we let $\calT_{\frakp}(K/k)$ denote an inertia subgroup of $\Gal(K/k)$ at $\frakp$. When $k$ is a number field, we let $S_p$ denote the set of primes of $k$ lying above $p$, and for $\frakp\in S_p$, $\calT_{\frakp}^{\circ}(K/k)$ is a (non-canonical) subgroup of $\calT_{\frakp}(K/k)$ that will be defined in \S~\ref{S:notation}. For two pro-$p$ groups $G$ and $H$, we let $G \ast H$ denote the free pro-$p$ product of $G$ and $H$. 
	
	\begin{theorem}[Theorem~\ref{thm:pres}]\label{thm:main-2} 
		Let $k$ be a global field and $p$ a prime number not equal to $\Char(k)$. Let $K/k$ be a Galois extension of global fields such that $K$ is a subfield of $k_S^T(p)$ for some finite sets $S$ and $T$ of primes of $k$. Suppose there exists a subset $S' \subset S$ such that $G_{S'}^T(k)(p)=1$ and $\B_{S'\backslash T}^{S'\cup T}(k, \F_p)=0$. Then there exists a surjection of pro-$p$ groups
		\[
			\phi: \left(\Conv_{\frakp \in S\backslash (T\cup S' \cup S_p)} \calT_{\frakp}(K/k) \right) \ast \left(\Conv_{\frakp \in (S \cap S_p) \backslash (T \cup S')} \calT_{\frakp}^{\circ}(K/k) \right) \longtwoheadrightarrow \Gal(K/k)
		\]
		such that $\Gal(K_{S'}^T(p)/k)$ is the quotient of the domain of $\phi$ modulo all the local relators at primes in $S\backslash T$.
	\end{theorem}
	
	Theorem~\ref{thm:main-2} is very useful for the study of the structure of $G_{S'}^T(K)(p)$. For instance, an immediate consequence is that, by applying the Kurosh subgroup theorem, one can obtain an upper bound for the generator rank of $G_{S'}^T(K)$. A formula of this upper bound is stated in Proposition~\ref{prop:upper}, and the following theorem about the $p$-rank of class group of $K$ is obtained by applying Proposition~\ref{prop:upper} to the cases \eqref{item:intro-Q} and \eqref{item:intro-Fq(t)} above, with $S'=\{\text{the archimedean prime of }\Q\}$ if $p=2$ and $k=\Q$, and with $S'=\O$ otherwise. Let $\Ram(K/k)$ denote the set of primes of $k$ that are ramified in $K/k$, and $\Ram^f(K/k)$ denote the set consisting of all the nonarchimedean primes in $\Ram(K/k)$. 
	
	\begin{theorem}[Corollaries~\ref{cor:Q-rank} and \ref{cor:Fq(t)-rank}]\label{thm:main-3}
		\begin{enumerate}
			\item Let $K/\Q$ be a finite Galois $p$-extension, and $\Cl^+(K)$ denote the narrow class group of $K$. Then 
			\begin{eqnarray*}
				\dim_{\F_p}\Cl^+(K)[p] &\leq& (\# \Ram^f(K/\Q) -1) [K:\Q] + 1 \\
				&& - \sum_{\frakp \in \Ram^f(K/\Q) \backslash S_p} \frac{[K:\Q]}{|\calT_{\frakp}(K/\Q)|} - \sum_{\frakp \in \Ram(K/\Q) \cap S_p} \frac{[K:\Q]}{|\calT^{\circ}_{\frakp}(K/\Q)|}.
			\end{eqnarray*}
			\item Assume $q$ is a prime power such that $p \nmid q$. Let $K/\F_q(t)$ be a finite Galois $p$-extension that is splitting completely at $\infty$. Then
			\[
				\dim_{\F_p}\Cl(K)[p] \leq (\# \Ram(K/\F_q(t)) -1) [K:\F_q(t)] + 1
				 - \sum_{\frakp \in \Ram(K/\F_q(t))} \frac{[K:\F_q(t)]}{|\calT_{\frakp}(K/\F_q(t))|}.
			\]
		\end{enumerate}
	\end{theorem}
	
	The above upper bound is known in some cases. For a quadratic extension $K/\Q$, the equality above holds, and it is well-known by the Gauss genus theory \cite{Gauss}. For a cyclic degree-$p$ extension $K/k$, a similar upper bound is given by Kl\"uners and Wang \cite{Kluners-Wang}. For a tamely ramified totally real multiquadratic extension $K/\Q$, the same upper bound is given by Koymans and Pagano \cite{Koymans-Pagano-genus}.
	
	Theorem~\ref{thm:main-2} can also be applied to find the lower bounds for the generator rank of $G_{S'}^T(K)$, by studying all the possibilities of local relators. Such a lower bound very much depends on the structure of $\Gal(K/k)$, and we give in Proposition~\ref{prop:lower} a lower bound that is not very good but works for all Galois groups $\Gal(K/k)$. In the case of $k=\Q$, to the best of the author's knowledge, the only known lower bound for the $p$-rank of class group of $K$ is given by the genus field (the maximal extension of $K$ which is obtained by composing $K$ with an abelian extension of $\Q$ and which is unramified over $K$). The lower bound in Proposition~\ref{prop:lower} is better than the one obtained from the genus field.
	
	To show the strength of Theorem~\ref{thm:main-2}, we study the $p$-rank of class groups of cyclic degree-$p^d$ extensions of $\Q$ and 2-rank of class groups of multiquadratic extensions of $\Q$, for a fixed ramification type. A \emph{ramification type} is a tuple $(\Gamma, n, \{I_i\}_{i=1}^n, I_{\infty})$ where $\Gamma$ is a finite $p$-group, $I_i$'s are cyclic subgroups of $\Gamma$ that generate $\Gamma$, and $I_{\infty}$ is a subgroup of $\Gamma$ of order at most 2. For each ramification type, define $\calS(\Gamma, n, \{I_i\}_{i=1}^n, I_{\infty})$ to be the set of tamely ramified extensions $K/\Q$ such that $\Gal(K/\Q)\simeq \Gamma$, $K/\Q$ is ramified at $n$ nonarchimedean primes with inertia subgroups given by $I_i$'s, and the inertia subgroup of $K/\Q$ at the archimedean prime is $I_{\infty}$ (a thorough definition is given in Definition~\ref{def:not}). For the case of $\Gamma=C_{p^d}$ and the case of $\Gamma=(C_2)^{\oplus d}$, we study the following questions:
	\begin{enumerate}[label=(\Roman*)]
		\item \label{item:Q-I} Is there a better lower bound of $\dim_{\F_p}\Cl(K)[p]$ that works for all $K \in \calS(\Gamma, n, \{I_i\}_{i=1}^n, I_{\infty})$?
		\item \label{item:Q-II} Is the upper bound in Theorem~\ref{thm:main-3} sharp for fields in $\calS(\Gamma, n, \{I_i\}_{i=1}^n, I_{\infty})$?
	\end{enumerate}
	
	For the case $\Gamma=C_{p^d}$, we show in Theorem~\ref{thm:lb-cyclic} that the lower bound obtained from the genus field is sharp. Although we do not know the answer to Question~\ref{item:Q-II}, we show in Proposition~\ref{prop:ub-cyclic} that there exist infinitely many fields in $\calS(C_{p^d}, n, \{I_i\}_{i=1}^n, I_{\infty})$ such that 
	\[
		\dim_{\F_p} \Cl(K)[p] \geq \sum_{i=1}^n \frac{|\Gamma|}{|I_i|} -1.
	\]
	In this case, the upper bound in Theorem~\ref{thm:main-3} is $(n-1)|\Gamma|+1-\sum_{i=1}^n \frac{|\Gamma|}{|I_i|}$.
	
	For the case $\Gamma=(C_2)^{\oplus d}$, in Theorem~\ref{thm:lb-multiquad}, we give a better lower bound for $\dim_{\F_2} \Cl(K)[2]$, that is
	\begin{equation}\label{eq:lb-multiquad-intro}
		n-d+\sum_{i=2}^{d-1}\max\left\{0, (i-1)\cdot \binom{d}{i} +(n-d) \cdot \binom{d-1}{i-1} -n \cdot \binom{d-1}{i-2}- \binom{d}{i-\alpha_{\infty}} \right\}
	\end{equation}
	where $\alpha_{\infty}$ is 1 if $I_{\infty}=1$ and 2 if $I_{\infty}\simeq C_2$. Note that this lower bound is much better than the one from the genus theory, which is $n-d$. Koymans and Pagano studied Question~\ref{item:Q-II} in \cite{Koymans-Pagano}, and their method uses their previous work \cite{Koymans-Pagano-genus} and some combinatorial ideas from the work of Smith \cite{Smith}. In Proposition~\ref{prop:K-P}, we use Theorem~\ref{thm:main-2} to recover an important step in their work.

\subsection{Distribution of class groups}
	We choose these two types of $p$-groups as our examples because the distribution of $p$-part of the class groups for such $p$-extensions of $\Q$ has been actively studied in the recent years. For a quadratic extension $K$ of $\Q$, $\Cl(K)/2\Cl(K)$ is completely determined by the genus field. Gerth \cite{Gerth84} conjectured that, away from the part determined by genus theory, the distribution of the 2-part of $2\Cl(K)$ is predicted by a Cohen--Lenstra type of probability measures. Fouvry and Kl\"uners \cite{Fouvry-Kluners-1, Fouvry-Kluners-2} proved that the distribution of $(2\Cl(K))[2]$ agrees with Gerth's conjecture for both real and imaginary quadratic fields. 
	In \cite{Gerth86}, Gerth extended his conjecture to predict the distribution of $p$-part of class groups of cyclic degree-$p$ extensions of $\Q$ for an odd prime $p$. Klys \cite{Klys} proved that the distribution of $((1-\zeta)\Cl(K))[1-\zeta]$ (which is the analogue of $(2\Cl(K))[2]$ for odd $p$ and where $\zeta$ is a generator of $\Gal(K/\Q)$) agrees with this conjecture of Gerth. The remarkable work \cite{Smith} of Smith proves Gerth's conjecture for $p=2$. Developing Smith's idea, Koymans and Pagano \cite{Koymans-Pagano-Gerth} proved Gerth's conjecture for odd $p$, conditionally on GRH. In \cite{Smith2022}, Smith strengthened his method to prove an analogous result for cyclic degree-$p$ extensions of any base number field. There are works attempting to generalize the results for cyclic degree-$p$ extensions to other types of $p$-extensions, such like biquadratic extensions, multiquadratic extensions and cyclic $p$-extensions of larger degree (for example, see \cite{Fouvry-Koymans} and \cite{Fouvry-Koymans-Pagano}). 
	
	We expect that the method in this paper will provide another tool to tackle this type of problems. Let's first explain how to understand Gerth's conjecture using our results. Assume $K/\Q$ is a tamely ramified quadratic extension ramified at $d$ nonarchimedean primes $\{\frakp_1, \ldots, \frakp_d\}$. By Theorem~\ref{thm:main-2}, there exists a surjection 
	\[
		\pi: \Conv_{i=1}^d C_2 \longtwoheadrightarrow \Gal(K/\Q)=C_2,
	\]
	and $\Gal(K_{\O}(2)/\Q)$ is the quotient of $\Conv_{i=1}^d C_2$ modulo the local relators. At each $\frakp_i$, let $x_i$ denote the generator of the $i$-th copy of $C_2$ in the free product, then because $x_i$ has order 2 and $\frakp_i$ is odd, the local relators are in the form of 
	\[
		x_i^{\frakp_i} y_i x_i^{-1} y_i^{-1} = x_i y_i x_i^{-1} y_i^{-1}
	\]
	for some element $y_i$, and hence the relator must be contained in the commutator subgroup of $\Conv_{i=1}^d C_2$. Note that the abelianization of $\Conv_{i=1}^d C_2$ is isomorphic to $(C_2)^{\oplus d}$, its quotient modulo the local relator at the archimedean prime corresponds the genus field, and the 2-primary part of the class group of $K$ is the quotient of the abelianization of $\ker \pi$ modulo all the local relators. Because the local relator at $\frakp_i$ is completely determined by $y_i$, the lifts of Frobenius elements, if one want study the distribution of $2$-part of class group as $K$ varies among quadratic fields, it suffices to study the distribution of Frobenius elements at the ramified primes. Smith showed that such Frobenius elements distribute randomly, which gives rise to the probability measure in Gerth's conjecture.  
	
	For a finite $p$-group $\Gamma$ and a $\Gamma$-extension $K/\Q$, the argument in the preceding paragraph still works, and the surjection $\pi$ is determined by the ramification type of $K/\Q$. To understand the distribution of $p$-part of class group of $\Gamma$-extensions over $\Q$, one may wonder if the distribution of Frobenius elements at ramified elements is still random, and if there is any unknown constraints on those Frobenius elements. To understand these problems, answering Questions \ref{item:Q-I} and \ref{item:Q-II} is the first step, because if the bounds of $p$-rank of class groups given in this paper are not sharp, then there are some restrictions that we haven't taken into account. Moreover, if the Frobenius elements do distribute randomly, then one can construct a random group model using the presentation results in this paper to give the conjectural distribution of class groups, in the same vein as previous works on the non-abelian Cohen--Lenstra heuristics \cite{BBH-imaginary, BBH-real, LWZB, Liu-rou}.

\subsection{Embedding Problems and Presentations of $G_S^T(k)$}
	
	For a pro-$p$ group $G$, the dimension of cohomology groups $\dim_{\F_p}H^1(G, \F_p)$ and $\dim_{\F_p}H^2(G, \F_p)$ are the generator rank and relator rank of $G$, which are the minimal number of generators and the minimal number of relators used in a presentation of $G$ respectivley. Koch and Shafarevich showed that, for a global field $K$, a finite set $S$, and a prime number $p\neq \Char(k)$, $G_S(k)(p)$ is finitely presented, and they gave estimates of the ranks of $G_S(k)(p)$. The abelian group $\B_S(k, \F_p)$ was defined in their work in order to relate the cohomology of $G_S(k)$ to the Galois cohomology of local fields $k_{\frakp}$. Hajir, Maire and Ramakrishna further developed this cohomological method and applied it to some interesting algebraic number theoretical questions (see \cite{Maire96}, \cite{HMR20}, \cite{HMR21}, etc.).
	
	The explicit presentation of $G_S(k)(p)$ in certain cases was previous discussed by Koch \cite[Thm.11.16]{Koch}. The presentation result Theorem~\ref{thm:main-1} in this paper broadly extend Koch's theorem and utilizes different methods in the proof. Our proof uses embedding problems with restricted ramification, which ask whether a given global field extension can be embedded into a larger extension that satisfies the prescribed restrictions on ramification. To study the presentation of $G_S^T(k)(p)$, one only need to consider embedding problems associated to central group extensions with kernel $\F_p$. However, our embedding problem results work for any solvable group extensions, using the abelian group $\B_{S\backslash T}^{S \cup T}(k, A)$ defined in the author's previous work \cite{Liu2020}, where $A$ is any finite simple $G_k$-module and $\B_{S\backslash T}^{S \cup T}(k, A)$ is a generalization of $\B_S(k, \F_p)$ from work of Koch and Shafarevich. Although, unlike the pro-$p$ situation, the embedding problems associated to solvable group extensions do not directly give us information about the presentation of the maximal solvable quotient of $G_S(k)$, we expect that our embedding results can be applied to study the generalization of Cohen--Lenstra--Martinet and Gerth conjectures, and we plan to work on this in forthcoming work.
	
\subsection{Outline of the paper}

	In \S~\ref{S:notation}, we list the notation and background knowledge that will be used throughout the paper. In \S~\ref{S:embedding}, we discuss the embedding problems with restricted ramification. Although only the embedding problems associated to $p$-group extensions are used in later sections, we give our results in \S~\ref{S:embedding} in a generalized setting so that one can apply them to study embedding problems associated to solvable group extensions. The main result of \S~\ref{S:embedding} is Lemma~\ref{lem:embedding-ST}, which shows that, when some cohomological condition is satisfied, the embedding problem with restricted ramification is solvable if and only if it is locally solvable at a set of primes. The cohomological condition is stated using the finite abelian group $\B_{S\backslash T}^{S\cup T}(k, \F_p)$ defined by Shafarevich and Koch, and its generalization $\B_{S \backslash T}^{S\cup T}(k, A)$ defined in the author's previous work \cite{Liu2020} and \S~\ref{def:B}.
	
	In \S~\ref{S:pres}, we use our embedding problem result to study presentations of $G_S^T(k)(p)$, and prove Theorem~\ref{thm:main-1} and Theorem~\ref{thm:main-2}. In particular, we apply Theorem~\ref{thm:main-2} to the special cases \eqref{item:intro-Q} and \eqref{item:intro-Fq(t)}, and in Corollaries \ref{cor:Q} and \ref{cor:Fq(t)} we explicitly give a presentation of $\Gal(K_{\O}(p)/\Q)$ and $\Gal(K_{\O}^{S_{\infty}}(p)/\Q)$ when $K$ is a $p$-extension of $\Q$ or $\F_q(t)$. In \S~\ref{S:rank-general}, we explain how to use the presentation result Theorem~\ref{thm:main-2} to compute the generator rank of $G_{S'}^T(K)(p)$, and study upper bounds for this generator rank in \S~\ref{SS:rank-ub} and lower bounds in \S~\ref{SS:rank-lb}.
	
	Finally in \S~\ref{S:Cyclic} and \S~\ref{S:multiquad}, we study the $p$-rank of class groups of $C_{p^d}$-extensions and the $2$-rank of class groups of $(C_2)^{\oplus d}$-extensions of $\Q$ respectively. We give the definition of ramification type and outline our method at the beginning of \S~\ref{S:Cyclic}.
	
\subsection*{Acknowledgement}
	The author thanks Peter Koymans and Christian Maire for helpful comments on an early draft of this paper. The author is partially supported by NSF Grant DMS-2200541.
\section{Notation and preliminary}\label{S:notation}
	
	In this paper, groups are always finite groups and profinite groups, and subgroups are topologically closed subgroups. For a group $G$, we let $G^{\ab}$ denote the abelianization of $G$, and $[G,G]$ denote the commutator subgroup. For two elements $a, b \in G$, we denote $a^b:=bab^{-1}$ and $[a,b]:=aba^{-1}b^{-1}$. If $H$ is a group with a continuous $G$-action, then the semidirect product $H \rtimes G$ is the group with the underlying set $\{(h, g) \mid h \in H, g\in G\}$ and the multiplication defined by $(h_1, g_1)(h_2, g_2)= (h_1 g_1(h_2), g_1g_2)$. The wreath product $H \wr G$ is the semidirect product $(\prod_{g \in G} H) \rtimes G$ where each $\gamma \in G$ acts on $\prod_{g\in G} H$ by mapping the $g$-coordinate to the $\gamma g$-coordinate. For elements $x_1, x_2 , \ldots$ of $G$, we let $\langle x_1, x_2, \ldots \rangle$ denote the subgroup of $G$ generated by $x_1, x_2, \ldots$. 
	
	For a group $G$ and a prime number $p$, we let $G(p)$ denote the pro-$p$ completion of $G$. For a finitely generated pro-$p$ group, the generator rank $d(G)$ of $G$ is the minimal number of generators of $G$, the Frattini subgroup $\Phi(G)$ of $G$ is the subgroup $[G,G]G^p$, and there is an isomorphism $G/\Phi(G) \simeq (\F_p)^{\oplus d(G)}$. The Burnside's basis theorem says that, when $G$ is pro-$p$, a subgroup $H$ of $G$ is proper if and only if the image of $H$ under the quotient map $G \to G/\Phi(G)$ is a proper subgroup of $G/\Phi(G)$. For two pro-$p$ groups $G$ and $H$, we let $G \ast H$ denote the free pro-$p$ product of $G$ and $H$ (see \cite[Def.(4.1.1)]{NSW}).
	
	For a global field $k$, we let $G_k$ denote the absolute Galois group of $k$. For sets $S$ and $T$ of primes of $k$, let $k_S$ denote the maximal extension of $k$ that is unramified outside $S$, let $k_S^T$ denote the maximal subextension of $k_S/k$ that is splitting completely at primes in $T$, and denote $G_S(k):=\Gal(k_S/k)$ and $G_S^T(k):=\Gal(k_S^T/k)$. Similarly, we let $k_S(p)$ and $k_S^T(p)$ to be the maximal pro-$p$ subextensions of $k_S/k$ and $k_S^T/k$, and then $G_S(k)(p)=\Gal(k_S(p)/k)$ and $G_S^T(k)(p)=\Gal(k_S^T(p)/k)$.
	Let $\calO_k$ and $\calO_{k, S}$ denote the ring of integers and the ring of $S$-integers of $k$. Let $\Cl(k)$ and $\Cl_S(k)$ denote the class group and the $S$-class group of $k$. For a number field $k$, we let $S_p(k)$ denote the set of all primes of $k$ lying above $p$, let $S_{\R}(k)$ and $S_{\C}(k)$ denote the sets of all real and complex places of $k$ respectively, and let $S_{\infty}(k)$ denote $S_{\R}(k) \cup S_{\C}(k)$. For the function field case, there is a prime $\infty$ of $\F_p(t)$, and then for each extension $k/\F_p(t)$, we let $S_{\infty}(k)$ denote the set of all primes of $k$ lying above $\infty$. When the choice of $k$ is clear, we will omit $k$ and just write $S_{\R}$, $S_{\C}$, $S_{\infty}$ and $S_p$. For a global field extension $K/k$, we let $\Ram(K/k)$ denote the set of all primes of $k$ that is ramified in $K/k$, and when $k$ is a number field we let $\Ram^f(K/k)$ denote the set of all the nonarchimedean primes in $\Ram(K/k)$. For sets $S$ and $T$ of primes of $k$, let $K_S^T$ (resp. $K_S^T(p)$) denote the maximal extension (resp. maximal pro-$p$ extension) of $K$ that is ramified only at primes above $S$ and splitting completely at each prime above $T$.
	
	For a prime $\frakp$ of $k$, we let $k_{\frakp}$ denote the local completion of $k$ at $\frakp$ and let $U_{\frakp}$ denote the unit group of the ring of integers of $k_{\frakp}$. 
	Let $\calG_\frakp$ and $\calT_{\frakp}$ denote the absolute Galois group $\Gal(\overline{k_{\frakp}}/k_{\frakp})$ and the inertia subgroup of the extension $\overline{k_{\frakp}}/k_{\frakp}$ respectively. 
	We choose a distinguished prime $\frakP_K$ of $K$ lying above $\frakp$ for each Galois global field extension $K/k$, such that for any two Galois extensions $K$ and $L$ of $k$, the chosen primes $\frakP_K$ and $\frakP_L$ lie below the same prime of $\overline{k}$. Then we let $\calG_{\frakp}(K/k)$ and $\calT_{\frakp}(K/k)$ denote the decomposition subgroup and the inertia subgroup of $K/k$ at $\frakP_K$ respectively, and we call them \emph{the} decomposition subgroup and \emph{the} inertia subgroup of $K/k$ at $\frakp$. (Practically, the choice of $\frakP_K$ will be implicit.) Then there are natural embeddings $\calT_{\frakp} \hookrightarrow \calG_{\frakp}\hookrightarrow G_k$ defined by identifying $\calT_{\frakp}$ and $\calG_{\frakp}$ with $\calT_{\frakp}(\overline{k}/k)$ and $\calG_{\frakp}(\overline{k}/k)$ respectively. Throughout this paper, when we use embeddings $\calT_{\frakp} \hookrightarrow G_k$ and $\calG_{\frakp} \hookrightarrow G_k$, we always mean the embeddings defined as above.
	
	Let $k$ be a global field, $\frakp$ a prime of $k$, and $p$ a prime number such that $p\neq \Char(k)$. Let $\Nm(\frakp)$ denote the cardinality of the residue field of $k_{\frakp}$. Presentations of $\calG_\frakp(p)$ are well-known, and we define our notation as follows.

	\begin{enumerate}
		\item\label{item:case-1} If $p\nmid \Nm(\frakp)$, then by the result of Iwasawa \cite[Thm.(7.5.3)]{NSW}, $\calG_{\frakp}(p)$ has a pro-$p$ presentation
	\[
		\calG_{\frakp}(p) \simeq \langle t_{\frakp}, s_{\frakp} \mid s_{\frakp} t_{\frakp} s_{\frakp}^{-1} = t_{\frakp}^{\Nm(\frakp)} \rangle.
	\]
	This isomorphsm identifies $t_{\frakp}$ and $s_{\frakp}$ as two elements of $\calG_{\frakp}(p)$, so with a slight abuse of notation, we view $t_{\frakp}$ and $s_{\frakp}$ as elements of $\calG_{\frakp}(p)$. The isomorphism above is not unique, but we choose and then fix such an isomorphism, so that $t_{\frakp}$ and $s_{\frakp}$ are fixed elements of $\calG_{\frakp}(p)$ throughout the paper. In particular, $\calT_{\frakp}(p)$ is the subgroup of $\calG_{\frakp}(p)$ generated by $t_{\frakp}$. 	
		\item\label{item:case-2} Assume that $k$ is a number field. If $p \mid \Nm(\frakp)$ and $\mu_p \not\subset k_{\frakp}$, then $\calG_{\frakp}(p)$ is a free pro-$p$ group on $[k_{\frakp}: \Q_p]+1$ elements \cite[Thm.(7.5.11)]{NSW}.
		So the inertia subgroup of $\calG_{\frakp}(p)$, denoted by $\calT_{k_{\frakp}}(p)$ (note that this is not equal to $\calT_{\frakp}(p)$)
		 is a normal subgroup of $\calG_{\frakp}(p)$ with quotient isomorphic to $\Z_p$. We choose and then fix elements $t_{\frakp;\,i} \in \calG_{\frakp}(p)$ for $i=1, \ldots, [k_{\frakp}: \Q_p]$, such that $\calT_{k_{\frakp}}(p)$ is the normal subgroup generated by them. We define $\calT^{\circ}_{\frakp}$ to be the subgroup of $\calG_{\frakp}(p)$ generated by these elements $t_{\frakp;\,i}$'s. So $\calT^{\circ}_{\frakp}$ is a free pro-$p$ group on $[k_{\frakp}: \Q_p]$ generators, and its normal closure in $\calG_{\frakp}(p)$ is $\calT_{k_{\frakp}}(p)$. If $p \mid \Nm(\frakp)$ and $\mu_p \subset k_{\frakp}$, then $\calG_{\frakp}(p)$ is a Demu\v{s}kin group of rank $[k_{\frakp}: \Q_p]+2$. In this case, to the best of the author's knowledge, there is no result in the literature that gives a presentation of $\calG_{\frakp}(p)$ from which $\calT_{k_{\frakp}}(p)$ can be described using the generators. So we will avoid this case when we want to describe the local relators in a presentation of $G_S^T(k)$ in \S\ref{S:pres}.
	\end{enumerate}
	For a pro-$p$ Galois extension $K/k$, we let $t_{\frakp}(K/k)$, $s_{\frakp}(K/k)$ and $t_{\frakp;\, i}(K/k)$ be elements of $\Gal(K/k)$ that are defined as the images of $t_{\frakp}$, $s_{\frakp}$ and $t_{\frakp;\, i}$ under the composite map $\calG_{\frakp}(p) \rightarrow \calG_{\frakp}(K/k) \hookrightarrow \Gal(K/k)$. We let $\calT^{\circ}_{\frakp}(K/k)$ denote the image of $\calT_{\frakp}^{\circ}$ under the composite map. For convenience, we sometimes also denote $\calT_{\frakp}(K/k)$ by $\calT^{\circ}_{\frakp}(K/k)$ in the case that $p\nmid \Nm(\frakp)$, when we want to uniformly deal with $\calT_{\frakp}^{\circ}$ and $\calT_{\frakp}$ for all $\frakp$.

	For a finite $G_k$-module $A$, we denote $A':=\Hom(A, \overline{k}^{\times})$ and $A^{\vee}:=\Hom(A, \Q/\Z)$, and we let $k(A)$ denote the minimal trivializing extension of $k$ for $A$. 
	For a set $S$ of primes of $k$, we denote 
	\[
		\prod'_{\frakp \in S} H^1(\calG_{\frakp}, A):= \left\{ (f_{\frakp})_{\frakp \in S} \in \prod_{\frakp \in S} H^1(\calG_{\frakp}, A) \,\bigg{|}\, \text{$f_{\frakp}$ is unramified for all but finitely many primes in $S$}\right\}.
	\]
	For a $\calG_{\frakp}$-module $A$, we let $H^1_{nr}(\calG_{\frakp}, A)$ denote the kernel of the restriction map $H^1(\calG_{\frakp}, A) \to H^1(\calT_{\frakp}, A)^{\calG_{\frakp}}$, and note that this restriction map is surjective because $\calG_{\frakp}/\calT_{\frakp} \simeq \hat{\Z}$.

\section{Embedding problems with restricted ramification}\label{S:embedding}

\subsection{Definition and properties of $\B_S^T(k,A)$}	
	
	\begin{definition}\label{def:B}
		Let $k$ be a global field, $S_0 \subseteq S$ two sets of primes of $k$, and $A$ a finite simple $\F_p[G_k]$-module such that $p\neq \Char(k)$. We define $\B_{S_0}^S(k,A)$ to be the cokernel of the following map
		\[
			\prod'_{\frakp \in S_0} H^1(\calG_{\frakp}, A) \times \prod_{\frakp \not\in S} H_{nr}^1(\calG_{\frakp}, A) \hookrightarrow \prod'_{\frakp} H^1(\calG_{\frakp}, A) \overset{\sim}{\longrightarrow} \prod'_{\frakp} H^1(\calG_{\frakp}, A')^{\vee} \longrightarrow H^1(G_k, A')^{\vee}.
		\]
		Here the first map is the natural embedding, the second map is the isomorphism obtained by the local Tate duality theorem \cite[Thm.(7.2.6) and Thm.(7.2.17)]{NSW}, and the last map is the Pontryagin dual of the product of restriction map.
	\end{definition}
	
	When the module $A$ is $\F_p$, $\B_{S\backslash T}^{S\cup T}(k, A)$ is exactly $\B_{S\backslash T}^{S \cup T}(k)$ defined in \cite[Def.(10.7.8)]{NSW}, which is the Pontryagin dual of 
		\begin{equation}\label{eq:def-V}
			V_{S\backslash T}^{S \cup T}(k):= \left\{ a \in k^{\times} \mid a \in k_{\frakp}^{\times p} \text{ for } \frakp \in S \backslash T \text{ and } a \in U_{\frakp} k_{\frakp}^{\times p} \text{ for } \frakp \not \in S \cup T  \right\} / k_{\frakp}^{\times p}.
		\end{equation}
		
	In several results in this section, we assume that $\Gal(k(A)/k)$ is solvable, so that the Hasse principle for $k$ and $A$ holds in dimension 1 and 2. 
	
	\begin{lemma}\label{lem:B-es}
		Let $S$ and $T$ be two sets of primes of $k$. Let $A$ be a finite simple $\F_p[G_S^T(k)]$-module such that $\Gal(k(A)/k)$ is solvable and $p\neq \Char(k)$. Then there is an exact sequence
		\[ 
			H^1(G_k, A) \longrightarrow \bigoplus_{\frakp \not\in S \cup T} H^1(\calT_{\frakp}, A)^{\calG_{\frakp}} \oplus \prod'_{\frakp \in T} H^1(\calG_{\frakp}, A) \longrightarrow \B_{S \backslash T}^{S \cup T} (k, A) \longrightarrow 0,
		\]
		where the first map is the product of maps $H^1(G_k, A) \to H^1(\calG_{\frakp}, A)$ at $p\in T$ and the composite maps $H^1(G_k,A) \to H^1(\calG_{\frakp}, A) \to H^1(\calT_{\frakp}, A)^{\calG_{\frakp}}$ at $\frakp \not\in S \cup T$. 
	\end{lemma}

	\begin{proof}
		Consider the following diagram
		\[\begin{tikzcd}
			 & \bigoplus\limits_{\frakp \not\in S \cup T} H^1(\calT_\frakp, A)^{\calG_{\frakp}} \times \prod'\limits_{\frakp \in T} H^1(\calG_{\frakp}, A) & & \\
			 H^1(G_k, A) 
			 \arrow[hook]{r} & \prod'\limits_{\frakp} H^1(\calG_{\frakp}, A) \arrow[two heads]{u} \arrow[two heads]{r} & H^1(G_k, A')^{\vee} \arrow[equal]{d} & \\
			 & \prod\limits_{\frakp \not \in S \cup T} H^1_{nr} (\calG_{\frakp}, A) \times \prod'\limits_{\frakp \in S \backslash T} H^1(\calG_{\frakp}, A) \arrow["\alpha"]{r} \arrow[hook]{u} &H^1(G_k, A')^{\vee}  \arrow[two heads]{r} & \B_{S \backslash T}^{S\cup T} (k, A). 
		\end{tikzcd}\]

		The exactness of the middle row follows from the long exact sequence of Poitou--Tate \cite[Thm.(8.6.10)]{NSW} together with the fact that $\Sha^1(G_k, A)=\Sha^2(G_k,A)=0$ \cite[Thm.(9.1.15)(i) and Cor.(9.1.16)(i)]{NSW}. The lower row comes from the definition of $\B_{S\backslash T}^{S\cup T}(k, A)$. 
		The exactness of the second column follows by the inflation-restriction exact sequence for local cohomology at each prime together with $H_{nr}^2(\calG_p, A)=0$ (as $\calG_{\frakp}/\calT_{\frakp}\simeq \hat{\Z}$). So by the snake lemma, we have 
		\[
			0 \longrightarrow \ker\alpha \longrightarrow H^1(G_k, A) \longrightarrow \bigoplus_{\frakp \not\in S \cup T} H^1(\calT_{\frakp}, A)^{\calG_{\frakp}} \times \prod'_{\frakp \in T} H^1(\calG_{\frakp}, A) \longrightarrow \B_{S\backslash T}^{S\cup T} (k, A) \longrightarrow 0,
		\]
		so the lemma follows.
	\end{proof}

	\begin{lemma}\label{lem:basic-B}
		\begin{enumerate}
			\item\label{item:basic-B-1} If $S_1 \subseteq S_2$, then there is a natural surjection $\B_{S_1\backslash T}^{S_1\cup T}(k,A) \twoheadrightarrow \B_{S_2\backslash T}^{S_2 \cup T}(k,A)$.
			
			\item\label{item:basic-B-2} If $T_1 \subseteq T_2$, then there is a natural surjection $\B_{S\backslash T_2}^{S\cup T_2}(k,A) \twoheadrightarrow \B_{S\backslash T_1}^{S \cup T_1}(k,A)$.
			
			\item\label{item:basic-B-3} If $T$ is finite, then $\B_{S\backslash T}^{S\cup T}(k,A)$ is finite.
			
		\end{enumerate}
	\end{lemma}
	
	\begin{proof}
		The statements \eqref{item:basic-B-1} and \eqref{item:basic-B-2} follow directly from Definition~\ref{def:B}. Consider the following diagram.
		\[\begin{tikzcd}
			\prod\limits_{\frakp \not \in S \cup T} H^1_{nr} (\calG_{\frakp}, A) \times \prod'\limits_{\frakp \in S \backslash T} H^1(\calG_{\frakp}, A) \arrow[hook]{d} \arrow{r} & H^1(G_k, A')^{\vee} \arrow[equal]{d} \arrow[two heads]{r} & \B_{S\backslash T}^{S\cup T}(k, A) \arrow[two heads]{d}\\ 
			\prod\limits_{\frakp \not \in S} H^1_{nr} (\calG_{\frakp}, A) \times \prod'\limits_{\frakp \in S} H^1(\calG_{\frakp}, A) \arrow[two heads]{d} \arrow{r} & H^1(G_k, A')^{\vee} \arrow[two heads]{r} & \B_{S}^{S}(k, A) \\ 
			\prod\limits_{\frakp \in T\backslash S} H^1_{nr}(\calG_{\frakp},A) \times \prod'\limits_{\frakp \in S \cap T} H^1(\calG_{\frakp}, A) & &
		\end{tikzcd}\]
		Since for each prime $\frakp$ of $k$, the local cohomology groups $H^1_{nr}(\calG_{\frakp}, A)$ and $H^1(\calG_{\frakp}, A)$ are both finite, so the lower-left term in the diagram above is finite. By the argument in the proof of \cite[Cor.8.8]{Liu2020}, $\B_S^S(k,A)$ is finite for any $S$. Therefore, the statement \eqref{item:basic-B-3} follows by applying the snake lemma to the above diagram.
	\end{proof}

		Next, we show that there are infinitely many choices of $S$ and $T$ such that $\B_{S\backslash T}^{S\cup T}(k, A)=0$.
		
		\begin{lemma}\label{lem:B=0}
			Let $T$ be a finite set of primes of $k$, and $A$ a finite simple $\F_p[G_k]$-module such that $\Gal(k(A)/k)$ is solvable and $p\neq \Char(k)$. For any finite set $T'$ of primes of $k$, there are infinitely many finite sets $S$ such that $S \cap T'=\O$ and $\B_{S\backslash T}^{S \cup T}(k,A)=0$.
		\end{lemma}
		
		\begin{proof}
			Let $T'$ be an arbitrary finite set of primes of $k$. By \cite[Thm.(9.1.15)(ii)]{NSW}, the map
			\[
				H^1(G_k, A') \longrightarrow \prod_{\frakp \not \in T \cup T'} H^1(\calG_{\frakp}, A')
			\]
			is injective. So by Definition~\ref{def:B}, we have $\B_{S\backslash T}^{S \cup T}(k,A)=0$ with $S=\{\text{all primes of }k\}\backslash (T \cup T')$. 
			
			Suppose $S_1$ is a finite set such that $S_1\cap T' = \O$. We will show that there exists a finite larger set $S_2 \supset S_1$ such that 
			\begin{equation}\label{eq:cond-S2}
				\B_{S_2\backslash T}^{S_2\cup T}(k,A)=0 \quad \text{ and } \quad S_2 \cap T'=\O, 
			\end{equation}
			and then the lemma immediately follows. Consider the diagram
			\[\begin{tikzcd}
				\prod\limits_{\frakp \in S_1 \backslash (T \cup T')} H^1(\calG_{\frakp}, A) \times \prod\limits_{\frakp \not \in S_1 \cup T \cup T'} H^1_{nr}(\calG_{\frakp}, A) \arrow[hook]{d} \arrow{r}& H^1(G_k, A')^{\vee} \arrow[two heads]{r} \arrow[equal]{d}& \B_{S_1\backslash (T\cup T')}^{S_1 \cup T \cup T'} (k,A) \\
				\prod'\limits_{\frakp \not\in T \cup T'} H^1(\calG_{\frakp}, A) \arrow[two heads]{r} \arrow[two heads]{d} & H^1(G_k, A')^{\vee} & \\
				\bigoplus\limits_{\frakp \not \in S_1 \cup T \cup T'} H^1(\calT_{\frakp}, A)^{\calG_{\frakp}} & & 
			\end{tikzcd}\]
			where the upper row comes from Definition~\ref{def:B} and the lower row is a surjection because of the discussion in the preceding paragraph. If $\B_{S_1 \backslash (T \cup T')}^{S_1 \cup T \cup T'}(k,A) \neq 0$, then by the snake lemma, there exists $\frakp \not\in S_1 \cup T \cup T'$ such that the image of $H^1(\calT_{\frakp}, A)^{\calG_{\frakp}}$ in $\B_{S_1 \backslash (T \cup T')}^{S_1 \cup T \cup T'}(k, A)$ is nontrivial, and hence the kernel of 
			\[
				\B_{S_1 \backslash (T \cup T')}^{S_1\cup T \cup T'}(k, A) \longtwoheadrightarrow \B_{(S_1 \cup \{\frakp \}) \backslash (T \cup T')}^{S_1 \cup \{\frakp\} \cup T \cup T'}(k, A)
			\]
			is nontrivial. By Lemma~\ref{lem:basic-B}\eqref{item:basic-B-3}, after repeating the procedure above for finitely many times, we obtain a finite set $S_2 \supset S_1$ such that $\B_{S_2 \backslash (T \cup T')}^{S_2 \cup T \cup T'}(k, A)=0$ and $S_2 \cap T'= \O$. Then \eqref{eq:cond-S2} follows by Lemma~\ref{lem:basic-B}\eqref{item:basic-B-2}.
		\end{proof}

		For the case $A=\F_p$, the following lemma gives an explicit method to determine when $\B_{S\backslash T}^{S \cup T}(k, \F_p)$ is zero.
		
		\begin{lemma}\label{lem:B=0-empty}
			The group $\B_{S\backslash T}^{S\cup T}(k,\F_p)$ is $0$ when the followings hold:
			\begin{enumerate}
				\item\label{item:B=0-empty-1} the $p$-part of the $S\cup T$-class group $\Cl_{S\cup T}(k)$ of $k$ is trivial, and
				\item\label{item:B=0-empty-2}  the following map is injective
				\begin{equation}\label{eq:map-O}
					\calO_{k,T}^{\times}/\calO_{k,T}^{\times p} \longrightarrow \prod_{\frakp \in S\backslash T}U_{\frakp}/U_{\frakp}^p
				\end{equation}
				where the map is induced by the natural embedding $\calO_{k,T}^{\times}\hookrightarrow U_{\frakp}$ for $\frakp \not\in T$.
			\end{enumerate}
		\end{lemma}
		
		\begin{proof}
			The group $\B_{S\backslash T}^{S \cup T}(k, \F_p)$ is the Pontryagin dual of $V_{S\backslash T}^{S\cup T}(k)= W/k^{\times p}$ with 
			\[
				 W:=\left\{ a\in k^{\times} \mid a \in k_{\frakp}^{\times p} \text{ for } \frakp \in S \backslash T \text{ and } a \in U_{\frakp}k_{\frakp}^{\times p} \text{ for } \frakp \not \in S \cup T \right\}.
			\]
			Consider the map
			\begin{eqnarray*}
				W &\longrightarrow& \Cl_{S\cup T}(k)[p] \\
				a &\longmapsto& [I]
			\end{eqnarray*}
			where $[I]$ is the class of an ideal $I$ of $\calO_k$ with $I^p\equiv(a)$ modulo the product of primes in $S \cup T$. This is a well-defined group homomorphism because $\Cl_{S\cup T}(k)$ is the quotient of $\Cl(k)$ modulo the classes of primes in $S \cup T$ and it induces a homomorphism $f: V_{S\backslash T}^{S\cup T}(k) \to \Cl_{S\cup T}(k)[p]$. One can see by definition of $W$ that the kernel of $f$ is exactly the kernel of the map \eqref{eq:map-O} defined in the lemma. Therefore, $V_{S\backslash T}^{S \cup T}(k)=0$ if the conditions \eqref{item:B=0-empty-1} and \eqref{item:B=0-empty-2} both hold, so the lemma follows.
		\end{proof}

\subsection{Embedding problems}

	Let $k$ be a global field.
	An \emph{embedding problem over $k$} is given by a surjection $G_k \twoheadrightarrow G$ and an exact sequence $1 \to A \to \widetilde{G} \to G \to 1$.
	\begin{equation}\label{eq:ep}
	\begin{tikzcd}
		& & &G_k \arrow[two heads, "\varphi"]{d} \arrow[dashed, "\psi", swap]{ld}& \\
		1 \arrow{r} & A \arrow{r} & \widetilde{G} \arrow["\rho"]{r} & G \arrow{r} & 1
	\end{tikzcd}
	\end{equation}
	The embedding problem is \emph{solvable (resp. properly solvable)} if there exists a map (resp. surjection) $\psi$ as describe by the dashed arrow, and the map $\psi$ is called a \emph{solution} to the embedding problem.
	At each prime $\frakp$ of $k$, by taking the restriction $\varphi_{\frakp}:= \varphi|_{\calG_{\frakp}}$, \eqref{eq:ep} induces a local embedding problem:
	\begin{equation}\label{eq:ep-local}
	\begin{tikzcd}
		& & &\calG_{\frakp} \arrow["\varphi_{\frakp}"]{d}  \arrow[dashed, "\psi_{\frakp}", swap]{ld}& \\
		1 \arrow{r} & A \arrow{r} & \widetilde{G} \arrow["\rho"]{r} & G \arrow{r} & 1.
	\end{tikzcd}
	\end{equation}
	We let $\cHom_G(G_k, \widetilde{G})$ and $\cHom_G(\calG_{\frakp}, \widetilde{G})$ denote the set of all solutions to \eqref{eq:ep} and \eqref{eq:ep-local} respectively.
	When $A$ is abelian, the conjugation action of $\widetilde{G}$ on $A$ induces a $G$-action on $A$, and then $A$ can be viewed as a $G_k$-module via $\varphi$. 
	
	The following is the local-global principle for embedding problems that we will use later to study the embedding problems with restricted ramification. Such local-global principles are known in many similar situations, but we cannot find in literature the following version that we need, so we give a complete proof which mostly just repeats the argument in \cite[Lem.(9.5.6)]{NSW}.
	
	\begin{lemma}\label{lem:embedding}
		If the kernel $A$ is a finite simple $G_k$-module such that $\Gal(k(A)/k)$ is solvable and $p\neq \Char(k)$, then 
		\[
			\cHom_{G}(G_k, \widetilde{G}) \neq \O \qquad \Longleftrightarrow \qquad \prod_{\frakp}\cHom_{G}(\calG_{\frakp}, \widetilde{G}) \neq \O,
		\]
		where on the right-hand side the product is over all the primes of $k$ (archimedean and non-archimedean). 
	\end{lemma}
	
	\begin{proof}
		The maps $\varphi$ and $\varphi_\frakp$ induce inflation maps
		\[
			\varphi^*: H^2(G, A) \to H^2(G_k, A) \quad \text{and} \quad \varphi^*_\frakp: H^2(G, A) \to H^2(\calG_{\frakp}, A).
		\]
		Let $c$ denote the class of $H^2(G,A)$ corresponding to the central extension $\rho:\widetilde{G} \to G$. By \cite[Lem.(3.5.9)]{NSW}, the embedding problem \eqref{eq:ep} (resp. \eqref{eq:ep-local}) is solvable if and only if the class $\varphi^*(c) =0$ (resp. $\varphi^*_\frakp(c)=0$). Because $\varphi_{\frakp}$ is obtained from $\varphi$ by the natural embedding $\calG_{\frakp}\hookrightarrow G_k$, the canonical homomorphism 
		\begin{equation}\label{eq:res}
			\phi: H^2(G_k, A) \longrightarrow \prod_{\frakp} H^2(\calG_{\frakp}, A)
		\end{equation}
		is compactible with $\varphi^*$ and $\varphi^*_\frakp$, i.e., $\phi \circ \varphi^*= \prod_{\frakp} \varphi^*_{\frakp}$. 
		So it suffices to show that $\phi$ is injective.
		When $p\neq \Char(k)$, $\phi$ is injective by \cite[Cor.(9.1.16)(i)]{NSW}. 
	\end{proof}

	\begin{lemma}\label{lem:embedding-ST}
		Let $S$ and $T$ be two sets of primes of $k$. Assume that the map $\varphi$ in the embedding problem \eqref{eq:ep} factors through $G_S^T(k)$, $A$ is a finite simple $\F_p[G_S^T(k)]$-module such that $\Gal(k(A)/k)$ is solvable, $p\neq \Char(k)$, and $\B_{S\backslash T}^{S\cup T}(k,A)=0$. We let $\cHom_G(G_S^T(k),\widetilde{G})$ denote the set of solutions $\psi$ to \eqref{eq:ep} that factors through $G_S^T(k)$. Then 
		\[
			\cHom_G(G_S^T(k), \widetilde{G}) \neq \O \quad \Longleftrightarrow \quad \prod_{\frakp \in S\backslash T} \cHom_G(\calG_{\frakp}, \widetilde{G}) \neq \O.
		\]
	\end{lemma}
	
	\begin{proof}
		Because $\varphi$ factors through $G_S^T(k)$, $\varphi$ is unramified at each $\frakp \not \in S$ and is split completely at $\frakp \in T$. Since the maximal unramified extension of $k_{\frakp}$ is of Galois group $\hat{\Z}$ over $k_{\frakp}$,  for each $\frakp \not \in S \backslash T$, the map $\varphi_{\frakp}: \calG_{\frakp} \to G$ can be lifted to an unramified homomorphism $\calG_{\frakp} \to \widetilde{G}$, so $\cHom_G(\calG_{\frakp}, \widetilde{G})\neq \O$. So by Lemma~\ref{lem:embedding}, the $\cHom_G(G_k, \widetilde{G})\neq \O$ if and only if $\prod_{\frakp \in S \backslash T} \cHom_G(\calG_{\frakp}, \widetilde{G})\neq \O$,
		then it suffices to show that $\cHom_G(G_k, \widetilde{G})\neq \O$ if and only if $\cHom_G(G_S^T(k), \widetilde{G})\neq \O$. The ``if'' direction is obvious as $G_S^T(k)$ is a quotient of $G_k$.
		For the ``only if'' direction, assume $\psi \in \cHom_G(G_k, \widetilde{G})$, and we will construct a 1-cocycle $f: G_k \to A$ such that the group homomorphism
		\begin{eqnarray}
			^f\psi: G_k & \longrightarrow& \widetilde{G} \label{eq:xpsi}\\
			\sigma &\longmapsto& f(\sigma)\psi(\sigma) \nonumber
		\end{eqnarray}
		factors through $G_S^T(k)$.
		
		Let $\psi_{\frakp}: \calG_\frakp \to \widetilde{G}$ denote the composition of $\calG_\frakp \hookrightarrow G_k$ and $\psi$. For $\frakp \in T$, since $\varphi_{\frakp}$ is the trivial map, we have that $A$ is a trivial $\calG_{\frakp}$-modue and the image of $\psi_{\frakp}$ is contained in $A$, so $\psi_{\frakp}$ is a group homomorphism from $\calG_\frakp$ to $A$; in this case, we define $f_{\frakp}:=(\psi_{\frakp})^{-1} \in H^1(\calG_{\frakp}, A)$. Next, suppose $\frakp \not \in S\cup T$, and we define $g_{\frakp}$ to be the homomorphism $\psi_{\frakp}|_{\calT_{\frakp}}: \calT_{\frakp} \to \widetilde{G}$. Because $\varphi_{\frakp}$ is unramified, we see that $g_{\frakp}$ is a homomorphism from $\calT_{\frakp}$ to $A$, and moreover, it is $\calG_\frakp$-equivariant (the $\calG_{\frakp}$-action on $\calT_\frakp$ is conjugation). Therefore, $g_{\frakp}$ is an element in $\Hom_{\calG_{\frakp}}(\calT_{\frakp}, A)=H^1(\calT_{\frakp},A)^{\calG_{\frakp}}$. We define $f_{\frakp}:= g_{\frakp}^{-1}$ for each $p \not\in S\cup T$. Since all the $f_{\frakp}$ are obtained by restriction of the global homomorphism $\psi$, only finitely many of them are nonzero when restricted to inertia, and hence we obtain 
		\[
			\prod_{\frakp \not \in S\backslash T} f_{\frakp} \in \bigoplus_{\frakp \not \in S \cup T} H^1(\calT_{\frakp}, A)^{\calG_{\frakp}} \times \prod'_{\frakp \in T} H^1(\calG_{\frakp}, A).
		\] 
		By the assumption $\B_{S\backslash T}^{S \cup T}(k, A)=0$ and Lemma~\ref{lem:B-es}, there exists $f\in H^1(G_k, A)$ such that at $\frakp \in T$ the local restriction of $f$ is $f_{\frakp}$ and at $\frakp \not \in S \cup T$ the restriction of $f$ induced by $\calT_{\frakp}\hookrightarrow \calG_{\frakp} \hookrightarrow G_k$ is $f_\frakp$. Then one can check by our construction that the map $^f\psi$ defined in \eqref{eq:xpsi} is a group homomorphism from $G_k$ to $\widetilde{G}$ such that it is a lift of $\varphi$, unramified at $\frakp \not\in S$ and split completely at $\frakp \in T$. Therefore, we obtain an element $^f\psi$ of $\cHom_G(G_S^T(k), \widetilde{G})$, so we proved that if $\cHom_G(G_k,\widetilde{G})\neq \O$ then $\cHom_G(G_S^T(k), \widetilde{G})\neq \O$. 
	\end{proof}
	
	By Lemma~\ref{lem:B=0-empty}, for any odd prime $p$, we have $\B_{\O}^{\O}(\Q, \F_p)=0$ and therefore $\B_S^S(\Q, \F_p)=0$ for any $S$. So Lemma~\ref{lem:embedding-ST} says that the solvability of an embedding problem from $G_S(\Q)$ is equivalent to the solvability of local embedding problems at primes in $S$. However, when considering a nontrivial $G_{\Q}$-module $A$, we usually should not expect $\B_{\O}^{\O}(\Q, A)$ to be zero.
	
	\begin{example}
		The quadratic field $K=\Q(\sqrt{-23})$ has class number 3. Let $L$ denote the Hilbert class field of $K$. Then $\Gal(L/\Q)\simeq S_3$ and we have the following embedding problem
		\[\begin{tikzcd}
			& & & G_{\{23, \infty\}}(\Q) \arrow[two heads]{d} \arrow[dashed]{dl} & \\
			1 \arrow{r} & A \arrow{r} & C_9 \rtimes C_2 \arrow{r} & \Gal(L/\Q) \arrow{r} & 1,
		\end{tikzcd}\]
		where the action of $C_2$ on $C_9$ is taking inversion and $A$ is the nontrivial $C_2$-module that is isomorphic to $C_3$. Because an inertia subgroup of $L/\Q$ at 23 is isomorphic to $C_2$, by the structure of local Galois group of tamely ramified extension, a decomposition subgroup at $23$ is also isomorphic to $C_2$, so the induced local embedding problem at $23$ is solvable, and each solution maps $\calG_{23}$ to a subgroup of $C_9\rtimes C_2$ which is isomorphic to $C_2$. Similarly, the local embedding problem at $\infty$ is also solvable. So, if the global embedding problem is solvable, then the solution defines a $C_9$-extension of $K$ that is unramfied, which contradicts to $|\Cl(K)|=3$.
		So Lemma~\ref{lem:embedding-ST} implies that $\B_{\{23, \infty\}}^{\{23, \infty\}}(\Q, A)\neq 0$. By Lemma~\ref{lem:basic-B}, we obtain $\B_{\O}^{\O}(\Q, A) \neq 0$.
	\end{example}

\section{Pro-$p$ completion of $G_S^T(k)$ and its presentation}\label{S:pres}

	In this section, we formulate and prove the main presentation results: Theorems~\ref{thm:main-1} (Theorem~\ref{thm:pres-ST}) and \ref{thm:main-2} (Theorem~\ref{thm:pres}), and we will apply Theorem~\ref{thm:pres} to the cases of $k=\Q$ and $k=\F_q(t)$ as desribed in Corollaries~\ref{cor:Q} and \ref{cor:Fq(t)}. Let $k$ denote a global field. We define $\delta$ and $\delta_{\frakp}$ for $\frakp$ a prime of $k$ as follows
	\[
		\delta=\begin{cases}
			1 & \text{if $\mu_p \subset k$} \\
			0 & \text{otherwise}
		\end{cases} \quad \text{and} \quad 
		\delta_{\frakp}=\begin{cases}
			1 & \text{if $\mu_p \subset k_{\frakp}$} \\
			0 & \text{otherwise}
		\end{cases}.
	\]
	\begin{theorem}\label{thm:pres-ST}
		Let $k$ be a global field and $p$ a prime such that $p\neq \Char(k)$. Let $S$ and $T$ be two finite sets of primes of $k$. Assume $\B_{S\backslash T}^{S\cup T}(k,\F_p)=0$. 
		\begin{enumerate}
			\item \label{item:pres-ST-1}
			Let 
			\[
				d:=1+\sum_{\frakp \in S \backslash (T \cup S_{\C})} \delta_{\frakp} - \delta + \sum_{\frakp \in (S \backslash T) \cap S_p} [k_{\frakp}: \Q_p] - \# T\cup S_{\infty}.
			\]
			Then, denoting the free pro-$p$ group on $d$ generators by $F_d$, there is a surjection of pro-$p$ groups
			\[
				\varphi: F_d \longrightarrow G_S^T(k)(p).
			\]
			\item \label{item:pres-ST-2}
			For each $\frakp \in S \backslash (T \cup S_p \cup S_{\infty})$, we pick elements $x_{\frakp} \in \varphi^{-1}(t_{\frakp}(k_S^T/k))$ and $y_{\frakp} \in \varphi^{-1}(s_{\frakp}(k_S^T/k))$.
			When $p=2$, we pick $x_{\frakp}\in \varphi^{-1}(t_{\frakp}(k_S^T/k))$ for each $\frakp \in (S\cap S_{\R}) \backslash T$.
			If $(S\cap S_p)\backslash T=\O$ when $\mu_p \subset k$, then $\ker \varphi$ is the normal subgroup of $F_d$ generated by elements of the following types:
			\begin{eqnarray}
				(x_{\frakp})^{\Nm(\frakp)}y_{\frakp} x_{\frakp}^{-1} y_{\frakp}^{-1}&\quad& \text{for }\frakp \in S\backslash (T \cup S_{\infty}) \text{ such that } \delta_\frakp=1, \label{eq:r-type-1}\\
				x_{\frakp}^2 &\quad& \text{for } \frakp \in (S\cap S_{\R})\backslash T,\text{ if $p=2$}. \label{eq:r-type-2}
			\end{eqnarray}
			\item \label{item:pres-ST-3}
				Assume there exists $S' \subset S$ such that $G_{S'}^T(k)(p)=1$. Assume $(S \cap S_p)\backslash (T \cup S')=\O$ when $\mu_p \subset k$.
				Then $F_d$ is generated by the following elements:
			\begin{eqnarray}
				x_{\frakp} &\quad & \text{for } \frakp\in S\backslash(T \cup S_{\C} \cup S') \text{ such that } \delta_{\frakp}=1,  \label{eq:g-type-1}\\
				x_{\frakp;1}, \,x_{\frakp;2}, \,\ldots, \, x_{\frakp;[k_{\frakp}: \Q_p]} &\quad & \text{for } \frakp \in (S\cap S_p) \backslash (T \cup S'), \label{eq:g-type-2}
			\end{eqnarray}
			where in the last type, $x_{\frakp;i}$ is an element in $\varphi^{-1}(t_{\frakp;i}(k_S^T/k))$ for $i=1, \ldots, [k_{\frakp}: \Q_p]$. Moreover, if $\B_{S'\backslash T}^{S'\cup T}(k,\F_p)=0$, then $F_d$ is freely generated by these elements.
		\end{enumerate}
	\end{theorem}
	
	\begin{remark}\label{rmk:pres-ST}
		In \eqref{item:pres-ST-2}, when $\mu_p \subset k$, we assume $(S\cap S_p)\backslash T =\O$ only because the local relators at primes above $p$ are not as simple as \eqref{eq:r-type-1} and \eqref{eq:r-type-2}. 
		Interested readers can read about the local relators at those primes from \cite[pages. 416-417]{NSW}.
		Without this assumption, it is still true that $\ker \varphi$ is the normal subgroup generated by all the local relators at primes in $S\backslash T$, and one can verify it by following proof. 
	\end{remark}
	
	\begin{proof}
		By \cite[Thm.(10.7.12)]{NSW}, $G_S^T(k)(p)$ is a pro-$p$ group of generated by $d$ element, so there is a surjection $\varphi$ as defined in \eqref{item:pres-ST-1}. 
		By abuse of notation, we denote $t_{\frakp}(k_S^T/k)$, $s_{\frakp}(k_S^T/k)$ and $t_{\frakp;\, i}(k_S^T/k)$ by $t_{\frakp}$, $s_{\frakp}$ and $t_{\frakp;\, i}$ respectively.
		Note that by the assumption that $(S\cap S_p)\backslash T= \O$ if $\mu_p\subset k$, we have that any $\frakp$ in \eqref{eq:r-type-1} is not in $S_p$. For any such $\frakp$, $k_S^T(p)/k$ is tamely ramified at $\frakp$, so by the relators of the local Galois group, we have that 
		$(t_{\frakp})^{\Nm(\frakp)} s_{\frakp} t_{\frakp}^{-1} s_{\frakp}^{-1}=1$ in $G_S^T(k)(p)$, and hence $(x_{\frakp})^{\Nm(\frakp)} y_{\frakp} x_{\frakp}^{-1} y_{\frakp}^{-1} \in \ker \varphi$. Similarly, when $p=2$, we have $x_{\frakp}^2 \in \ker \varphi$ for $\frakp \in (S \cap S_{\R}) \backslash T$.
		
		Let $N$ be the normal subgroup of $F_d$ generated by all elements described in \eqref{eq:r-type-1} and \eqref{eq:r-type-2}. By argument above, we see that $N\subseteq \ker \varphi$. Suppose $N\neq \ker \varphi$. 
		Then by the property of pro-$p$ groups \cite[Cor.(3.9.3)]{NSW}, there exists a normal subgroup $M$ of $F_d$ such that $N \subseteq M\subsetneq \ker \varphi$ and $\ker\varphi/M = \F_p$. Then after taking quotient modulo $M$, we obtain a group extension and hence an embedding problem
		\begin{equation}\label{eq:embedding-p}
		\begin{tikzcd}
			& & &G_S^T(k)(p) \arrow[dashed]{dl} \arrow[equal]{d}& \\
			1 \arrow{r} & \F_p \arrow{r} & F_d/M \arrow{r} & G_S^T(k)(p) \arrow{r}& 1.
		\end{tikzcd}
		\end{equation}
		By the structure of tamely ramified local Galois group, a prime $\frakp \not \in S_p \cup S_{\infty}$ can ramified in a pro-$p$ extension only if $\Nm(\frakp) \equiv 1$ mod $p$, i.e., only if $\delta_{\frakp}=1$. Recall we showed in the proof of Lemma~\ref{lem:embedding-ST} that the local embedding problem at $\frakp$ induced by \eqref{eq:embedding-p} is always solvable when $\frakp$ is unramified in $k_S^T(p)/k$. So by Lemma~\ref{lem:embedding-ST}, the embedding problem~\eqref{eq:embedding-p} is solvable if and only if its induced local embedding problem is solvable at each prime in the following cases:
		\begin{enumerate}[label=(\roman*)]
			\item\label{item:c-1} $\frakp \in S\backslash (T \cup S_p \cup S_\infty)$ with $\delta_{\frakp}=1$,
			\item\label{item:c-2} $\frakp \in (S\cap S_p)\backslash T$,
			\item\label{item:c-3} $\frakp \in (S\cap S_{\R})\backslash T$. 
		\end{enumerate}
		For each prime $\frakp$ in Case~\ref{item:c-1}, let $\overline{x}_{\frakp}$ and $\overline{y}_{\frakp}$ denote images in $F_d/M$ of $x_{\frakp}$ and $y_{\frakp}$ respectively, then because $(\overline{x}_{\frakp})^{\Nm(\frakp)} \overline{y}_{\frakp} \overline{x}_{\frakp}^{-1} \overline{y}_{\frakp}^{-1}=1$ in $F_d/M$, the subgroup of $F_d/M$ generated by $\overline{x}_{\frakp}$ and $\overline{y}_{\frakp}$ is a solution to the local embedding problem at $\frakp$. For a prime $\frakp$ in Case~\ref{item:c-2}, since $\mu_p\not\subset k$, the local absolution pro-$p$ Galois group $\calG_{\frakp}(p)$ is free, so the local embedding problem is always solvable. Finally, for a prime $\frakp$ in Case~\ref{item:c-3}, the image of $x_{\frakp}$ in $F_d/M$ has order dividing 2, so it gives a solution at $\frakp$. Thus, \eqref{eq:embedding-p} is solvable.
		
		Recall that $G_S^T(k)(p)$ is a pro-$p$ group with generator rank $d$, so by Burnside's basis theorem, the group extension (buttom row) in \eqref{eq:embedding-p} is non-split. In other words, we showed that, when we suppose $N \neq \ker \varphi$, we can embed $k_S^T(p)$ into a larger field in $k_S^T(p)$, which cannot happen. So the statement \eqref{item:pres-ST-2} follows.
		
		Next, we prove the statement \eqref{item:pres-ST-3}. Because of the assumption $G_{S'}^T(k)(p)=1$, we see that $G_S^T(k)(p)$ is generated by the inertia subgroups at all primes of $k_S^T(p)$ lying above $S\backslash (T \cup S')$. Let $\Phi$ denote the Frattini subgroup of $G_S^T(k)(p)$. Then, by Burnside's basis theorem, those inertia subgroups generate $G_S^T(k)(p)$ if and only if their images generate the Frattini quotient $G_S^T(k)(p)/\Phi$. Because $G_S^T(k)(p)/\Phi$ is abelian (the conjugation action is trivial), it follows that $G_S^T(k)(p)/\Phi$ is generated by \emph{the} inertia subgroups at all primes of $k$ in $S \backslash (T \cup S')$ (here \emph{the} inertia subgroup means the distinguished inertia subgroup defined in \S~\ref{S:notation}). For each $\frakp \not \in S_{\C} \cup S_{p}$ (including the case that $\frakp \in S_{\R}$ and $p=2$), the inertia subgroup of $\calG_{\frakp}(p)$ is cyclic if $\delta_{\frakp}=1$ and is trivial otherwise. For $\frakp \in (S\cap S_{p})\backslash (T\cup S')$, because of the assumption $\mu_p \not \in k$, $\calG_{\frakp}(p)^{\ab}$ is the free abelian group generated by $[k_\frakp:\Q_p]+1$ elements and the inertia subgroup inside $\calG_{\frakp}(p)^{\ab}$ is generated by $[k_{\frakp}: \Q_p]$ elements. Therefore, we showed that $G_S^T(k)(p)$ is generated by the elements $t_{\frakp}$ for $\frakp \in S \backslash (T \cup S_{\C} \cup S')$ with $\delta_{\frakp}=1$ and the elements $t_{\frakp; 1},\, \ldots\, , t_{\frakp;[k_{\frakp}: \Q_p]}$ for $\frakp \in (S \cap S_p) \backslash (T \cup S')$. So $F_d$ is generated by elements in \eqref{eq:r-type-1} and \eqref{eq:r-type-2}.
		When $G_{S'}^T(k)(p)=1$ and $\B_{S'\backslash T}^{S'\cup T}(k,\F_p)=0$, by comparing the generator ranks of $G_{S'}^T(k)(p)$ and $G_S^T(k)(p)$ using \cite[Thm.(10.7.12)]{NSW}, it follows that the generator rank of $G_S^T(k)(p)$ is 
		\begin{equation}\label{eq:g-rank}
			\sum_{\frakp \in S\backslash (T \cup S_{\C} \cup S')} \delta_{\frakp} + \sum_{\frakp \in (S \cap S_p) \backslash (T \cup S')} [k_{\frakp}: \Q_p].
		\end{equation}
		Note that the number of elements in \eqref{eq:g-type-1} and \eqref{eq:g-type-2} is exactly equal to \eqref{eq:g-rank}, so those elements form a minimal generator set of $F_d$.
	\end{proof}
	
	\begin{theorem}\label{thm:pres}
		Let $K/k$ be a Galois extension of global fields such that $\Gal(K/k)$ is a pro-$p$ group with $p\neq \Char(k)$. Assume the following conditions are satisfied.
		\begin{itemize}
			\item $S$ and $T$ are finite sets of primes of $k$ such that $K$ is contained in $k_S^T(p)$. 		
			\item $S' \subset S$ such that $G_{S'}^T(k)(p)=1$. 
		\end{itemize}
		Denote
		\[
			R:=\left\{\frakp \in S\backslash (T \cup S_{\C} \cup S') \mid \delta_{\frakp}=1\right\} \cup \left((S\cap S_p) \backslash (T \cup S')\right).
		\]
		\begin{enumerate}
			\item\label{item:pres-1} Let
			\[
				\phi: \left(\Conv_{\frakp \in R\backslash S_p} \calT_{\frakp}(K/k)\right) \, \ast\,  \left(\Conv_{\frakp \in R\cap S_p} \calT^{\circ}_{\frakp}(K/k)\right)\longrightarrow \Gal(K/k),
			\]
			denote the homomorphism mapping the factor $\calT_{\frakp}(K/k)$ (resp. $\calT^{\circ}_{\frakp}(K/k)$) in the free product identically to the inertia subgroup $\calT_{\frakp}(K/k)$ (resp. $\calT^{\circ}_{\frakp}(K/k)$) of $K/k$ at $\frakp$. Then $\phi$ is a surjection and it factors through $\Gal(K_{S'}^T(p)/k)$.
			\item\label{item:pres-2} Assume $(S \cap S_p) \backslash T=\O$ when $\mu_p \subset k$, and assume $\B_{S'\backslash T}^{S'\cup T}(k, \F_p)=0$. The Galois group $\Gal(K_{S'}^T(p)/k)$ is the quotient of $\phi$ modulo elements of the following types:
			\begin{eqnarray}						(x_{\frakp})^{\Nm(\frakp)} y_{\frakp} x_{\frakp}^{-1} y_{\frakp}^{-1} &\quad & \text{for $\frakp \in S\backslash (T \cup S_{\infty})$ such that $\delta_{\frakp}=1$} \label{eq:pres-r-type-1} \\
				x_{\frakp}^2& \quad & \text{for $\frakp \in (S\cap S_{\R}) \backslash T$, if $p=2$}, \label{eq:pres-r-type-2}
			\end{eqnarray}
			where $x_{\frakp}$ and $y_{\frakp}$ are preimages of $t_{\frakp}(K_{S'}^T(p)/k)$ and $s_{\frakp}(K_{S'}^T(p)/k)$ respectively.
		\end{enumerate}
	\end{theorem}
	
	\begin{proof}
		For convenience, we denote $\calT^{\circ}_{\frakp}(K/k):=\calT_{\frakp}(K/k)$ for every $\frakp \in R\backslash S_p$. By the universal property of the free product, the homomorphism $\phi$ is well-defined. 
		Because $K_{S'}^T(p)/K$ does not further ramify at primes not in $S'$, we have $\calT^{\circ}_{\frakp}(K_{S'}^T(p)/k)\simeq \calT^{\circ}_{\frakp}(K/k)$ for every $\frakp \in R$. So similarly, there is a homomorphism $\varphi'$ from the domain of $\phi$ to $\Gal(K_{S'}^T(p)/k)$, which $\varphi$ factors through. Note that $K_{S'}^T(p)/k$ can ramify only at primes in $R\cup S'$. 
		
		The map $\varphi'$ is surjective. Indeed, if $\varphi'$ is not surjective, then $\im \varphi'$ is contained in a proper maximal subgroup of $\Gal(K_{S'}^T(p)/k)$, where the maximal subgroup has to be normal by the property of pro-$p$ groups. Then the quotient of $\Gal(K_{S'}^T(p)/k)$ modulo that maximal subgroup gives a nontrivial $C_p$-extension of $k$ that is ramified only at primes in $S'$, which contradicts to the assumption $G_{S'}^T(k)(p)=1$. So the statement~\eqref{item:pres-1} then follows.
		
		For the rest of the proof, we prove the statement \eqref{item:pres-2}. Because $\B_{S' \backslash T}^{S'\cup T}(k, \F_p)=0$ implies $\B_{S \backslash T}^{S\cup T}(k, \F_p)=0$ by Lemma~\ref{lem:basic-B}\eqref{item:basic-B-1}.
		We consider the diagram		
		\begin{equation}\label{eq:fiber}\begin{tikzcd}
			F_d \arrow["\pi"]{d} \arrow[two heads, "\varphi"]{r} & G_{S}^T(k)(p) \arrow[two heads]{d} \arrow[two heads]{dr} &\\
			\Conv\limits_{\frakp \in R} \calT^{\circ}_{\frakp}(K/k) \arrow[two heads,"\varphi'"]{r} & \Gal(K_{S'}^T(p) /k ) \arrow[two heads]{r} &\Gal(K/k).
		\end{tikzcd}\end{equation}
		The top row is the surjection $\varphi$ in Theorem~\ref{thm:pres-ST}. The left vertical map $\pi$ is defined by sending the generators of $F_d$ (described in \eqref{eq:g-type-1} and \eqref{eq:g-type-2}) at each prime $\frakp \in R$ to the factor $\calT^{\circ}_{\frakp}(K/k)$ in the free product. The triangle on the right is induced by $K_{S'}^T(p) \hookrightarrow k_{S}^T(p)$.

		\emph{Claim: } $\Gal(K_{S'}^T(p)/k)=F_d/(\ker \varphi \ker \pi)$.
		
		\emph{Proof of the claim:} First, it's clear that both $\ker \varphi$ and $\ker \pi$ are contained in the kernel of $F_d \to \Gal(K_{S'}^T(p)/k)$. Let $H$ denote $F_d/(\ker \varphi \ker \pi)$, and consider the quotient map $H \twoheadrightarrow \Gal(K_{S'}^T(p)/k)$. Then $G_{S}^T(k)(p) \twoheadrightarrow \Gal(K_{S'}^T(p)/k)$ factors through $H$, and the inertia subgroup of $G_{S}^T(k)(p)$ at a prime $\frakp$ in $R$ is mapped to the image in $H$ of the factor $\calT^{\circ}_{\frakp}(K/k)$ in the free product. Thus, $H$ corresponds to an extension of $K_{S'}^T(p)$ in $k_{S}^T(p)$ that is unramified at all primes above $R$. Recall that in the proof of Theorem~\ref{thm:pres-ST}, we classified all primes that could ramify in $k_{S}^T(p)/k$, and from which it follows that $K_{S'}^T(p)/K$ is the maximal subextension of $k_{S}^T(p)/K$ that is ramified only at all primes above $S'$. Therefore, $H=\Gal(K_{S'}^T(p)/k)$ and we proved the claim.
		
		By the claim, we have $\ker \varphi' = \pi(\ker \varphi)$. When $\B_{S' \backslash T}^{S' \cup T}(k, \F_p)=0$, the map $\pi$ in the diagram above is surjective by Theorem~\ref{thm:pres-ST}\eqref{item:pres-ST-3}. Then the statement \eqref{item:pres-2} in the theorem follows by results about $\ker\varphi$ in Theorem~\ref{thm:pres-ST}\eqref{item:pres-ST-2}, and by the surjectivity of $\pi$.
	\end{proof}
	
	\begin{corollary}\label{cor:Q}
		Let $K/\Q$ be a finite Galois $p$-extension. Then there is a surjection
		\begin{eqnarray*}
			\phi: \left(\Conv_{\frakp \in \Ram^f(K/\Q)\backslash S_p} \left\langle x_{\frakp} \, \Big{|}\, x_{\frakp}^{|t_{\frakp}(K/\Q)|} \right\rangle \right) \ast \calT^{\circ}_p(K/\Q)&\longrightarrow& \Gal(K_{S_{\infty}}(p)/\Q) \\
		\text{defined by} \quad	x_{\frakp} &\longmapsto& t_{\frakp}(K_{S_{\infty}}(p)/\Q) \\
			\text{and} \quad \calT^{\circ}_p(K/\Q) &\overset{=}{\longrightarrow}& \calT^{\circ}_p(K/\Q)
		\end{eqnarray*}
		For each $\frakp \in \Ram^f(K/\Q)$, we choose $y_{\frakp} \in \phi^{-1}(s_{\frakp}(K_{S_{\infty}}(p)/\Q))$. 
		We choose $x_{\infty} \in \phi^{-1}(t_{\frakp}(K_{S_{\infty}}(p)/\Q)$ for $\frakp$ the unique prime in $S_{\infty}(\Q)$. 
If $K$ is tamely ramified when $p=2$, then $\ker \phi$ is the normal subgroup generated by the following elements:
		\begin{eqnarray*}
			(x_{\frakp})^{\Nm(\frakp)} y_{\frakp} x_{\frakp}^{-1} y_{\frakp}^{-1} &\quad& \text{for $\frakp \in \Ram^f(K/\Q)\backslash S_p$}, \\
			x_{\infty}^2 &\quad& \text{if $p=2$}.
		\end{eqnarray*}
			
	\end{corollary}
	
	\begin{remark}\label{rmk:Q}
		When $p=2$ and $K$ is totally real, the corollary gives a presentation of the narrow $2$-class tower group of $K$. Because $\Gal(K_{\O}(2)/\Q)$ is the quotient of $\Gal(K_{S_{\infty}}(2)/\Q)$ modulo the image of $x_{\infty}$, one can also apply the corollary to obtain a presentation of the $2$-class tower group of $K$. 
	\end{remark}
	
	\begin{proof}
		By Lemma~\ref{lem:B=0-empty}, $\B_{\O}^{\O}(\Q, \F_p)=0$ if $p$ is odd, and $\B_{S_{\R}(\Q)}^{S_{\R}(\Q)}(\Q, \F_2)=0$. We let 
		\begin{itemize}
			\item $S=\Ram(K/\Q)$, $S'=\O$ and $T=\O$, if $p$ is odd; and
			\item $S=\Ram(K/\Q) \cup S_{\R}(\Q)$, $S'=S_{\R}(\Q)$ and $T=\O$, if $p=2$.
		\end{itemize}
		Then $\B_{S'\backslash T}^{S'\cup T}(\Q, \F_p)=0$ and by Lemma~\ref{lem:basic-B}\eqref{item:basic-B-1} $\B_{S \backslash T}^{S\cup T}(\Q, \F_p)=0$. All the assumptions in Theorem~\ref{thm:pres} are satisfied, so the corollary follows by Theorem~\ref{thm:pres} and the fact that $\calT_{\frakp}(K/k)$ is a cyclic group of order $|t_{\frakp}(K/\Q)|$ for any $\frakp \in \Ram^f(K/\Q) \backslash S_p$.
	\end{proof}

	\begin{corollary}\label{cor:Fq(t)}
		Let $K/\F_q(t)$ be a finite $p$-extension with $p\nmid q$ that is splitting completely at $\infty$. Then there is a surjection
		\begin{eqnarray*}
			\phi: \Conv_{\frakp \in \Ram(K/\F_q(t))} \left\langle x_{\frakp} \, \Big{|} \, x_{\frakp}^{|t_{\frakp}(K/\F_q(t))|} \right\rangle &\longrightarrow& \Gal(K_{\O}^{S_{\infty}}(p)/\F_q(t)) \\
			x_{\frakp} & \longmapsto & t_{\frakp}(K_{\O}^{S_{\infty}}(p)/\F_q(t)).
		\end{eqnarray*}
		For each $\frakp \in \Ram(K/\F_q(t))$, we choose $y_{\frakp} \in \phi^{-1}(s_{\frakp}(K_{\O}^{S_{\infty}}(p)/\Q))$. Then $\ker \phi$ is the normal subgroup generated the following elements:
		\[
			(x_{\frakp})^{\Nm(\frakp)} y_{\frakp} x_{\frakp}^{-1} y_{\frakp}^{-1} \quad \quad \text{for } \frakp \in \Ram(K/\F_q(t)).
		\]
	\end{corollary}
	
	\begin{proof}
		Because $\B_{\O}^{S_{\infty}}(\F_q(t), \F_p)=0$ and $G_{\O}^{S_{\infty}}(\F_q(t))(p)=1$, the corollary following from Theorem~\ref{thm:pres} by applying $S=\Ram(K/\F_q(t))$, $S'=\O$ and $T=S_{\infty}(\F_q(t))$.
	\end{proof}

\section{Generator Rank of $G_{S'}^T(K)(p)$}\label{S:rank-general}

	In this section, we use the assumptions and notation in Theorem~\ref{thm:pres} (we do not assume the assumptions in Theorem~\ref{thm:pres}\eqref{item:pres-2} unless otherwise stated). We denote 
	\[
		\calF_{K/k, S, S', T}:= \left( \Conv_{ \frakp \in R \backslash S_p} \calT_{\frakp}(K/k) \right) \ast \left(\Conv_{ \frakp \in R \cap S_p} \calT^{\circ}_{\frakp}(K/k) \right).
	\]
	The surjection $\phi$ in Theorem~\ref{thm:pres}\eqref{item:pres-1} induces the following diagram in which each row is exact, and then we see that $d(G_{S'}^T(K)(p)) \leq d(\ker \phi)$. In Section~\ref{SS:rank-ub}, we will discuss the upper bound of $d(G_{S'}^T(K)(p))$ and give its applications in special cases.
	\[\begin{tikzcd}
		1 \arrow{r} & \ker \phi/\Phi(\ker \phi) \arrow{r} \arrow[two heads]{d} & \calF_{K/k, S, S', T}/\Phi(\ker \phi) \arrow{r} \arrow[two heads]{d} & \Gal(K/k) \arrow{r} \arrow[equal]{d} & 1\\
		1 \arrow{r} &G_{S'}^T(K)(p)/\Phi(G_{S'}^T(K)(p)) \arrow{r} & \Gal(K_{S'}^T(p)/k)/\Phi(G_{S'}^T(K)(p)) \arrow{r} & \Gal(K/k) \arrow{r} & 1
	\end{tikzcd}\]
	
	If we assume $(S \cap S_p)\backslash T = \O$ when $\mu_p \subset k$, Theorem~\ref{thm:pres}\eqref{item:pres-2} gives a description of generators of $\ker( \ker \phi \to G_{S'}^T(K)(p))$, and we will use it to give a lower bound on $d(G_{S'}^T(K)(p))$ in Section~\ref{SS:rank-lb}.

\subsection{Upper bound}\label{SS:rank-ub}

	\begin{proposition}\label{prop:upper}
		With the assumptions and notation in Theorem~\ref{thm:pres}, if $[K:k]$ is finite, then we have the following upper bound for the generator rank of $G_{S'}^T(K)(p)$
		\[
			d(G_{S'}^T(K)(p)) \leq  (\# R-1) [K:k] - \sum_{\frakp \in R\backslash S_p}\frac{[K:k]}{|\calT_{\frakp}(K/k)|} - \sum_{\frakp \in R\cap S_p}\frac{[K:k]}{|\calT^{\circ}_{\frakp}(K/k)|} +1.
		\]
		Moreover, assuming $(S  \cap S_p) \backslash T= \O$ when $\mu_p \subset k$, and assuming $\B_{S'\backslash T}^{S' \cup T}(k, \F_p)=0$, the equality holds if and only if elements in \eqref{eq:pres-r-type-1} and \eqref{eq:pres-r-type-2} are contained in the Frattini subgroup of $\ker \phi$.
	\end{proposition}
	
	\begin{remark}
		\begin{enumerate}
			\item Recall that the choice of $\calT^{\circ}_{\frakp}$ for $\frakp \in R \cap S_p$ is not canocial. The above upper bound depends on the choice of $\calT^{\circ}_{\frakp}(K/k)$, because different choices of $t_{\frakp; i}$ could give $\calT^{\circ}(K/k)$ with different sizes.
			\item When the assumption $G_{S'}^T(k)(p)=1$ in Theorem~\ref{thm:pres} is not satisfied, then one can still obtain an upper bound for $d(G_{S'}^T(K)(p))$ by applying the Kurosh subgroup theorem to a presentation defined by the composition $F_d \twoheadrightarrow G_S^T(k)(p) \twoheadrightarrow \Gal(K_{S'}^T(p)/k) \twoheadrightarrow \Gal(K/k)$ with $d=d(G_S^T(k)(p))$.
		\end{enumerate}
	\end{remark}
	
	\begin{proof}
		We apply the Kurosh subgroup theorem \cite[Thm.(4.2.1)]{NSW} to compute the generator rank of $\ker \phi$, and obtain 
		\[
			d(\ker \phi) = \sum_{\frakp \in R}\left( [K:k] -\frac{[K:k]}{|\calT^{\circ}_{\frakp}(K/k)|}\right) -[K:k] +1.
		\]
		Because $G_{S'}^T(K)(p)$ is a quotient of $\ker \phi$, we have $d(G_{S'}^T(K)(p)) \leq d(\ker \phi)$, and hence we have the inequality in the proposition. By Theorem~\ref{thm:pres}\eqref{item:pres-2}, $G_{S'}^T(K)(p)$ is the quotient of $\ker \phi$ modulo the closed $\Conv_{\frakp \in R} \calT^{\circ}_{\frakp}(K/k)$-normal subgroup generated by elements in \eqref{eq:pres-r-type-1} and \eqref{eq:pres-r-type-2}. So by the Burnside's basis theorem, $d(\ker \phi)=d(G_{S'}^T(K)(p))$ if and only if $\ker (\ker \phi \to G_{S'}^T(K)(p))$ is contained in the Frattini subgroup of $\ker \phi$, and hence if and only if elements in \eqref{eq:pres-r-type-1} and \eqref{eq:pres-r-type-2} are contained in the Frattini subgroup of $\ker \phi$.
	\end{proof}
	
	\begin{corollary}\label{cor:Q-rank}
		Let $K/\Q$ be a finite Galois $p$-extension. Then
		\begin{eqnarray*}
			&&\dim_{\F_p} \Cl^+(K)[p] \\
			&\leq& (\#\Ram(K/\Q)-1)[K:\Q] - \sum_{\frakp \in \Ram(K/\Q)\backslash S_p} \frac{[K:\Q]}{|\calT_{\frakp}(K/\Q)|} - \sum_{\frakp \in \Ram(K/\Q)\cap S_p} \frac{[K:\Q]}{|\calT^{\circ}_{\frakp}(K/\Q)|} +1.
		\end{eqnarray*}
	\end{corollary}

	\begin{proof}
		Because $\dim_{\F_p}\Cl(K)[p]=\dim_{\F_p}\Cl(K)/p\Cl(K)=d(G_{\O}(K)(p))$ and $\dim_{\F_2}\Cl^+(K)[2]=\dim_{\F_2}\Cl^+(K)/2\Cl^+(K)=d(G_{S_\infty}(K)(2))$, the corollary follows from Lemma~\ref{prop:upper} by applying $S=\Ram(K/\Q)$, $S'=T=\O$ if $p$ is odd, and $S=\Ram(K/\Q) \cup S_{\R}(\Q)$, $S'=S_{\R}(\Q)$ and $T=\O$ if $p=2$.
	\end{proof}
	
	\begin{corollary}\label{cor:Fq(t)-rank}
		Let $K/\F_q(t)$ be a finite Galois $p$-extension with $p\nmid q$ that is splitting completely at $\infty$. Then 
		\[
			\dim_{\F_p} \Cl(K)[p] \leq (\#\Ram(K/\F_q(t))-1)[K:\F_q(t)] - \sum_{\frakp \in \Ram(K/\F_q(t))} \frac{[K:\F_q(t)]}{|\calT_{\frakp}(K/\F_q(t))|}+1.
		\]
	\end{corollary}
	
	\begin{proof}
		Note that $\dim_{\F_p} \Cl(K)[p] = \dim_{\F_p} \Cl(K)/p\Cl(K) = d(G_{\O}^{S_{\infty}}(K)(p))$. So the corollary follows from Lemma~\ref{prop:upper} by applying $S=\Ram(K/\F_q(t))$, $S'=\O$ and $T=S_{\infty}(\F_q(t))$.
	\end{proof}

\subsection{Lower bound}\label{SS:rank-lb}

	For a number field $K$, the  genus field of $K$ is the maximal extension of $K$ which is obtained by composing $K$ with an abelian extension of $\Q$ and which is unramified at all finite primes of $K$. Because all the abelian extensions of $\Q$ are classified by the Kronecker--Weber theorem, so one can explicitly write down the genus field and then give a lower bound of the class number of $K$. For a global field extension $K/k$ satisfying the conditions in Theorem~\ref{thm:pres}, we can apply the idea of the genus field to give a lower bound of $d(G_{S'}^T(K)(p))$.

	Consider the Frattini quotient map of $\calF_{K/k, S, S', T}$
	\[
		\psi:  \calF_{K/k, S, S', T} \longrightarrow \prod_{\frakp \in R\backslash S_p} C_p \, \times\,  \prod_{\frakp \in R\cap S_p} C_p^{\oplus d(\calT_{\frakp}^{\circ}(K/k))}.
	\]
	Denote $N:=\ker \psi \cap \ker \phi$ and assume $(S  \cap S_p)\backslash T= \O$ when $\mu_p \subset k$. Note that if $\frakp$ is ramified in the $p$-extension $K/k$ then $\Nm(\frakp)\equiv 1$ mod $p$. So the relators of the type \eqref{eq:pres-r-type-1} are all contained in the commutator subgroup of $\calF_{K/k, S, S', T}$. When $p=2$, the relators in \eqref{eq:pres-r-type-2} are in $\calF_{K/k, S, S', T}^2$. So all the relators in \eqref{eq:pres-r-type-1} and \eqref{eq:pres-r-type-2} are contained in $N$. 
	Therefore, $\ker(\ker \phi \to G_{S'}^T(K)(p))$ is contained in $N$, and hence
	\[
		d(G_{S'}^T(K)(p)) \geq \dim_{\F_p} \ker \phi/N
	\]
	Note that the above is an analogue of the genus theory for $\Q$, because, when $k=\Q$ and $S'=T=\O$, the extension of $\Q$ defined by the quotient $\calF_{K/k, S, S', T}/N$ is exactly the genus field of $K$.

	We can give a better lower bound for $d(G_{S'}^T(K)(p))$ by studying the structure of $\ker\phi/\Phi(\ker \phi)$ more carefully.
	 Denote $M:=[N, \calF_{K/k, S, S', T}] \Phi(\ker \phi)$.
	Then the extension
	\begin{equation}\label{eq:cge}
		1 \longrightarrow N/M \longrightarrow \calF_{K/k, S, S', T}/M \longrightarrow \calF_{K/k, S, S', T}/N \longrightarrow 1
	\end{equation}
	is a central extension and an extension of $\Gal(K/k)$ by an elementary $p$-group, and moreover, this extension group is maximal among all such quotients of $\calF_{K/k, S, S', T}$. We define
	\begin{equation}\label{eq:def-nu}
		\nu_{K/k, S, S', T}:=\min\left(0,  \dim_{\F_p} N/M - \#\{ \frakp \in S \backslash (T \cup S_{\C}) \mid \delta_{\frakp}=1\} \right).
	\end{equation}

	\begin{proposition}\label{prop:lower}
		With the assumptions and notation in Theorem~\ref{thm:pres}\eqref{item:pres-2}, if $[K:k]$ is finite, then we have the following lower bound for the generator rank of $G_{S'}^T(K)(p)$
		\[
			d(G_{S'}^T(K)(p)) \geq \#R\backslash S_p+\left(\sum_{\frakp \in R \cap S_p} d(\calT^{\circ}_{\frakp}(K/k))\right) - d(\Gal(K/k)) +\nu_{K/k, S, S', T}.
		\]
	\end{proposition}
	
	\begin{proof}
		The relators in \eqref{eq:pres-r-type-1} and \eqref{eq:pres-r-type-2} are exactly for the primes in $\{ \frakp \in S \backslash (T \cup S_{\C}) \mid \delta_{\frakp}=1\}$. Because \eqref{eq:cge} is central with elementary abelian-$p$ kernel, the image in $\calF_{K/k, S, S', T}/M$ of the normal subgroup of $\calF_{K/k, S, S', T}$ generated by each relator is either the trivial subgroup or a subgroup of $N/M$ isomorphic to $C_p$. So the cokernel of the projection map from $\ker(\ker \phi \to G_{S'}^T(K)(p))$ to $N/M$ has dimension at least $\nu_{K/k, S, S', T}$. 
		Then the inequality in the proposition follows by
		\begin{eqnarray*}
			&& \dim_{\F_p} \ker(\calF_{K/k, S, S', T}/N \to \Gal(K/k))\\
			&=& d(\calF_{K/k, S, S', T}) - d(\Gal(K/k)) \\
			&=& \#R\backslash S_p+\left(\sum_{\frakp \in R \cap S_p} d(\calT^{\circ}_{\frakp}(K/k))\right) - d(\Gal(K/k)).
		\end{eqnarray*}
	\end{proof}
	
	One could repeat the procedure above by taking the maximal elementary abelian-$p$ central extension of $\calF_{K/k, S, S', T}/M$ under $\calF_{K/k, S, S', T}$ to improve lower bound. However, how much the improvement could be depends on the structure of $\ker \phi / \Phi(\ker \phi)$ (as an $\F_p[\Gal(K/k)]$-module). For example, in Section~\ref{S:Cyclic}, we discuss the case that $k=\Q$, $S'=T=\O$ and $K/k$ is cyclic, and show that the lower bound is actually given by the genus field of $K$. In Section~\ref{S:multiquad}, we discuss the case that $k=\Q$, $S'=T=\O$ and $K/k$ is multiquadratic, and we give a lower bound much better than the one in Proposition~\ref{prop:lower}.

\section{$p$-rank of class groups of $C_{p^d}$-fields}\label{S:Cyclic}

	In this section and the next, we focus on tamely ramified extensions of $\Q$ of Galois group $C_{p^d}$ and $(C_2)^{\oplus d}$, and study the rank of $\Cl(K)[p]$ and $\Cl(K)[2]$ respectively when given a prescribed ramification type.

	\begin{definition}\label{def:not}
		Let $\Gamma$ be a finite group, $n$ a positive integer, $I_i$ for $i \in \{1,\ldots, n\}$ nontrivial subgroups of $\Gamma$ such that all of their conjugates generate $\Gamma$, and $I_{\infty}$ a subgroup of $\Gamma$ of order 1 or 2. Then we say a Galois extension $K/\Q$ is of ramification type $(\Gamma, n, \{I_i\}_{i=1}^n, I_{\infty})$ if there exist distinct nonarchimedean primes $\frakp_1, \ldots, \frakp_n$ of $\Q$ such that 
		\begin{itemize}
			\item $K/\Q$ is of Galois group $\Gamma$ and tamely ramified,
			\item $K/\Q$ is unramified at any nonarchimedean primes outside $\{\frakp_1, \ldots, \frakp_n\}$, and
			\item there exists an isomorphism $\Gal(K/\Q) \simeq \Gamma$ such that: for each $i$ one of the inertia subgroups at $\frakp_i$ is mapped to $I_i$, and one of the inertia subgroups at the archimedean prime is mapped to $I_{\infty}$. (Here we do not require the inertia subgroup that mapped to $I_i$ is the one at the distinguished prime.)
		\end{itemize}
		We let $\calS(\Gamma, n, \{I_i\}_{i=1}^n, I_{\infty})$ denote the set of all fields $K/\Q$ of ramification type $(\Gamma, n, \{I_i\}_{i=1}^n, I_{\infty})$.
	\end{definition}
	
	In the definition, we require that $I_i$, $1\leq i \leq n$ and their conjugates generate $\Gamma$, because otherwise the set $\calS(\Gamma, n, \{I_i\}_{i=1}^n, I_{\infty})$ is empty as $\Q$ does not have any nontrivial Galois extension that is unramified at every finite prime. 
	Given a ramification type, we consider a homomorphism
	\begin{equation}\label{eq:const-pi}
		\pi: \Conv_{i=1}^n I_i \longrightarrow \Gamma
	\end{equation}
	defined by mapping the component $I_i$ identically to the subgroup $I_i \subset \Gamma$. Then for any $K \in \calS(\Gamma, n, \{I_i\}_{i=1}^n, I_{\infty})$, the isomorphism $\Gal(K/\Q)\simeq \Gamma$ identifies $\pi$ with the surjection $\phi$ associated to $K/\Q$ defined in Theorem~\ref{thm:pres}. As we discussed in Section~\ref{S:rank-general}, to study the rank of $\Cl(K)[p]$ (especially its lower bound), it is crucial to understand the module structure of $\ker \pi/ \Phi(\ker\pi)$. Then, since the ramification type already determines all the inertia subgroups, we need to understand all the possibilities of lifts of Frobenius elements in $\Conv I_i$, and use Corollary~\ref{cor:Q} to estimate the $p$-rank of the class group.

	We will study the case of $\Gamma=C_{p^d}$ in this section, and study the case of $\Gamma=(C_2)^{\oplus d}$ in Section~\ref{S:multiquad}.
	
\subsection{Module structure}\label{SS:ms-cyclic}

	Consider a ramification type $(\Gamma, n, \{I_i\}_{i=1}^n, I_{\infty})$ for the case that $\Gamma$ is a cyclic $p$-group, i.e. $\Gamma=C_{p^d}$.
	Without loss of generality, we assume $I_1=\Gamma$. 
	
	We let $\gamma$ denote a generator of $\Gamma$. For each $i$, we let $x_i$ denote a generator of $I_i$ and $a_i$ denote a generator of the cyclic $\F_p[\Gamma]$-module $\F_p[\Gamma]/ (\sum_{g \in I_i} g)$. Then we define a map
	\begin{eqnarray}
		\alpha: \Conv_{i=1}^n I_i & \longrightarrow & \left(\bigoplus_{i=2}^n \F_p[\Gamma] / (\sum_{g \in I_i} g) \right) \rtimes \Gamma \label{eq:alpha}\\
		\text{defined by}\quad x_1 &\longmapsto& (0, \gamma) \nonumber\\
		x_i &\longmapsto& (a_i, \pi(x_i)) \text{ for }i \geq 2. \nonumber
	\end{eqnarray}
	Because $a_i$ is annihilated by $\sum_{g\in I_i} g$, we have 
	\[
		\alpha(x_i)^{|I_i|}=(a_i, \pi(x_i))^{|I_i|} = \left(a_i + \pi(x_i)(a_i)+ \cdots + \pi(x_i^{|I_i|-1})(a_i), \pi(x_i^{|I_i|}) \right) =1,
	\] 
	so $\alpha$ is a well-defined homomorphism.
	
	\begin{lemma}\label{lem:cyc-structure}
		The homomorphism $\alpha$ is surjective and $\ker \alpha = \Phi(\ker\pi)$. So $\alpha$ induces an isomorphism
		\[
			\rho: \, \left( \Conv_{i=1}^n I_i \right)/\Phi(\ker \pi) \overset{\sim}{\longrightarrow}  \left(\bigoplus_{i=2}^n \F_p[\Gamma] / (\sum_{g \in I_i} g) \right) \rtimes \Gamma.
		\]
	\end{lemma}
	
	\begin{proof}
		The surjectivity of $\alpha$ follows by the observation that $(0, \gamma)$ and all the $(a_i, \pi(x_i))$ generate the codomain of $\alpha$. By definition of $\alpha$, one can check that $\pi$ is the composition of $\alpha$ and the natural projection
		\[
			\beta:  \left(\bigoplus_{i=2}^n \F_p[\Gamma] / (\sum_{g \in I_i} g) \right) \rtimes \Gamma \longrightarrow \Gamma.
		\]
	
	Because $\ker \beta$ is an elementary abelian $p$-group, we have $\ker \alpha \supseteq \Phi(\ker \pi)$. On the other hand, by the Kurosh subgroup theorem, $\ker \pi$ is a free profinite group of rank 
	\begin{equation}\label{eq:rank-comp}
		\sum_{i=1}^n \left(|\Gamma| - \frac{|\Gamma|}{|I_i|}\right) - |\Gamma|+1= \sum_{i=2}^n \left( |\Gamma| - \frac{|\Gamma|}{|I_i|}\right).
	\end{equation}
	Then, since the $\F_p$-rank of $\ker \beta$ equals \eqref{eq:rank-comp}, we have that $\ker \alpha = \Phi(\ker \pi)$. 
	\end{proof}
	
\subsection{Lower bound}

	We consider the Frattini quotient map 
	\[
		\psi: \Conv_{i=1}^n I_i \longrightarrow \bigoplus_{i=1}^n C_p
	\]
	defined by modulo the Frattini subgroup $\Phi(\Conv_{i=1}^n I_i)$. The images $\psi(x_1), \psi(x_2), \ldots, \psi(x_n)$ give a basis of $\im \psi$.

	\begin{lemma}\label{lem:lb-cyclic}
		If $y_2, y_3, \ldots, y_n$ are elements in $\Conv_{i=1}^n I_i$ such that 
		\[
			\langle \psi(y_i) \rangle = \langle \psi(x_1) \rangle \text{ for each $i \geq 2$}
		\]
		as subgroups of $\im \psi$, then the images of 
		\[
			[x_i, y_i] \text{ for all $i\geq 2$}
		\]
		under the map $\alpha$ in \eqref{eq:alpha} generate the commutator subgroup of $\im \alpha$.
	\end{lemma}
	
	\begin{proof}
		We let 
		\[
			\ab: \im \alpha =\left( \bigoplus_{i=2}^n \F_p[\Gamma]/ (\sum_{g \in I_i} g) \right) \rtimes \Gamma \longrightarrow \bigoplus_{i=2}^n C_p \times \Gamma
		\]
		denote the abelianization map.
		For each $2 \leq i \leq n$, we denote $M_i:= \F_p[\Gamma]/(\sum_{g \in I_i} g)$, and let 
		\[
			\psi_i:  M_i \longrightarrow \F_p
		\]
		denote the natural projection of $\Gamma$-modules (this projection is unique up to an isomorphism of $\F_p$). Then we have 
		\[
			\ker \ab= \bigoplus_{i=2}^n \ker \psi_i.
		\]
		
		Note that the kernel of the natural projection $\F_p[\Gamma] \to \F_p$ is generated by the group ring element $1-\gamma$, and is the unique maximal ideal of the ring $\F_p[\Gamma]$. Because $M_i$ is a quotient of $\F_p[\Gamma]$, we have that $\ker \psi_i$ is generated by the image of $1-\gamma$, so it is a cyclic module. Then, since $\Gamma$ is a $p$-group, the only simple $\F_p[\Gamma]$-module is $\F_p$, so 
		\[
			\ker \ab/\Rad(\ker \ab) = \left(\bigoplus_{i=2}^n \ker \psi_i \right)/ \Rad \left(\bigoplus_{i=2}^n \ker \psi_i \right)= \bigoplus_{i=2}^n \ker \psi_i / \Rad( \ker \psi_i)= \bigoplus_{i=2}^n \F_p
		\] 
		where $\Rad$ is for the radical of a module. By the Nakayama lemma, to show that the images of $[x_i, y_i]$ generate $\ker \ab$, it suffices to show that their images generate $\ker \ab / \Rad(\ker \ab)$.
		Thus, we consider the following quotient map
		\[
			\overline{\alpha}: \Conv_{i=2}^n I_i \longrightarrow \im \alpha / \Rad(\ker \ab) = \im \alpha / \bigoplus_{i=2}^n \Rad( \ker \psi_i) = \left( \bigoplus_{i=2}^n M_i / \Rad(\ker \psi_i) \right) \rtimes \Gamma
		\]
		and we want to show that the images of $[x_i, y_i]$ generate the commutator subgroup of the right-hand side. Note that 
		\[
			1 \longrightarrow \ker \ab/ \Rad(\ker \ab) \longrightarrow \im\overline{\alpha}  \longrightarrow \im \ab \longrightarrow 1
		\]
		is a central group extension.
		
		For $i\geq 2$, let $\overline{a}_i$ denote the image of $a_i \in M_i$ in $M_i /\Rad(\ker \psi_i)$, and we write $(b_i, \sigma_i)$ for $\overline{\alpha}(y_i)$, with $b_j \in \oplus_{j=2}^n M_j/\Rad( \ker \psi_j)$ and $\sigma_i \in \Gamma$. The condition $\langle \psi(y_i) \rangle = \langle \psi(x_1) \rangle$ implies that 
		\begin{enumerate}[label=(\roman*)]
			\item $b_i$ belongs to $\ker \ab / \Rad (\ker \ab)$, and
			\item\label{item:cond-gen} $\sigma_i$ is a generator of $\Gamma$.
		\end{enumerate}
		So we have
		\begin{eqnarray*}
			\overline{\alpha}([x_i, y_i]) &=& [(\overline{a}_i, \pi(x_i)), (b_i, \sigma_i)] \\
			&=& \left(\overline{a}_i -\sigma_i(\overline{a}_i) -b_i +\pi(x_i)(b_i), 1 \right) \\
			&=& \left( \overline{a}_i - \sigma_i(\overline{a}_i), 1 \right),
		\end{eqnarray*}
		where the second equality follows by the multiplication rule of wreath products, and the last equality uses the fact that $b_i \in \ker \ab/ \Rad(\ker \ab)$ is in the center of $\im \overline{\alpha}$.
		Because $a_i$ is a generator of $M_i$, $\overline{a}_i$ is a generator of $M_i/ \Rad(\ker \psi_i)$, so it follows by \ref{item:cond-gen} that $\overline{\alpha}([x_i, y_i])$ generates the component 
		$\ker \psi_i / \Rad(\ker \psi_i)$ of $\ker \ab/\Rad(\ker \ab)$.
		Since the above is true for all $2\leq i \leq n$, we see that all $\overline{\alpha}([x_i, y_i])$ for $2\leq i \leq n$ generate $\ker \ab / \Rad (\ker \ab)$. So we finish the proof by the argument in the preceding paragraph.
	\end{proof}
	
	For a finite prime $\frakp$ of $\Q$, by class field theory, if $\Nm(\frakp)\equiv 1 \mod p^m$, then there exists a unique field $E_{p^m}(\frakp)$ such that $\Gal(E_{p^m}(\frakp)/\Q)\cong C_{p^m}$ and $\Ram^f(E_{p^m}(\frakp)/\Q)=\{\frakp\}$. 
	
	\begin{lemma}\label{lem:recip}
		Let $\frakp$ and $\frakq$ be two finite primes of $\Q$, and assume $\Nm(\frakp)\equiv 1 \mod p$. Then $\frakq$ splits completely in $E_{p}(\frakp)/\Q$ if and only if $\frakp$ splits completely in $\Q(\mu_p, \sqrt[p]{\frakq})$.
	\end{lemma}
	
	\begin{proof}
		By the Artin reciprocity law, $\frakq$ splits completely in $E_{p}(\frakp)/\Q$ if and only if $\frakq$ is a $p$th power in $\F_\frakp$. The latter is equivalent to $\Q_\frakp(\sqrt[p]{\frakq})=\Q_{\frakp}$. By the assumption $\Nm(\frakp)\equiv 1 \mod p$, the local completion of $\Q(\mu_p, \sqrt[p]{\frakq})$ at $\frakp$ is $\Q_{\frakp}(\sqrt[p]{\frakq})$. Finally, note that $\frakp$ splits completely in $\Q(\mu_{p}, \sqrt[p]{\frakq})$ if and only if the local completion at $\frakp$ of the extension $\Q(\mu_{p}, \sqrt[p]{\frakq})/\Q$ is the trivial extension, so we finish the proof.
	\end{proof}

	\begin{theorem}\label{thm:lb-cyclic}
		Let $\Gamma$ be a cyclic $p$-group. For any ramification type $(\Gamma, n, \{I_i\}_{i=1}^n, I_{\infty})$ and any $K \in \calS(\Gamma, n, \{I_i\}_{i=1}^n, I_{\infty})$, we have
		\[
			\dim_{\F_p} \Cl(K)[p] \geq \begin{cases}
				n-2 & \text{if $p=2$, $n\geq 2$, $I_{\infty}=1$ and $|I_i|=2$ for some $i$},\\
				n-1 & \text{otherwise}.
			\end{cases}
		\]
		The equality holds for infinitely many fields in $\calS(\Gamma, n, \{I_i\}_{i=1}^n, I_{\infty})$.
	\end{theorem}
	
	\begin{proof}
		Without loss of generality, we assume $|I_1| \geq |I_2| \geq \ldots \geq |I_n|$, so $I_1=\Gamma$. First, it follows by class field theory that the set $\calS(\Gamma, n, \{I_i\}_{i=1}^n, I_{\infty})$ is non-empty. Given a field $K$ of this ramification type, by Corollary~\ref{cor:Q} and Remark~\ref{rmk:Q}, $\Gal(K_{\O}(p)/\Q)$ is the quotient of $\Conv_{i=1}^n I_i$ modulo relators at primes in $\Ram^f(K/\Q)$ and $S_{\R}(\Q)$. We note that
		\[
			\Nm(\frakp) \equiv 1 \mod |\calT_{\frakp}(K/\Q)| \text{ for all } \frakp \in \Ram^f(K/\Q).
		\]
		So the relator $(x_{\frakp})^{\Nm(\frakp)} y_{\frakp} x_{\frakp}^{-1} y_{\frakp}^{-1}$ at a nonarchimedean prime is contained in the commutator subgroup of $\Conv_{i=1}^n I_i$. The relator at the archimedean prime is $x_{\infty}$ if $I_{\infty}=1$ and is $x_{\infty}^2$ if $I_{\infty}\neq 1$. When $p=2$, $x_{\infty}^2$ is contained in the Frattini subgroup of $\Conv_{i=1}^n I_i$. Note that $(\Conv_{i=1}^n I_i)^{\ab} \simeq \prod_{i=1}^n I_i$ and each $I_i$ is cyclic, and the image of $x_{\infty}$ in this abelianization has order at most 2. If $p=2$ and $I_{\infty}=1$, the quotienting-out-$x_{\infty}$ action can reduce the generator rank by at most 1 only when $n\geq 2$ and $|I_i|=2$ for some $i$.
		So we see that 
		\begin{eqnarray*}
			&&\dim_{\F_p}\Gal(K_{\O}(p)/\Q)^{\ab}/(\Gal(K_{\O}(p)/\Q)^{\ab})^p \\
			&\geq&  \begin{cases}
				n-1 & \text{if $p=2$, $n\geq 2$, $I_{\infty}=1$ and $|I_i|=2$ for some $i$,}\\
				n & \text{otherwise.}
			\end{cases}
		\end{eqnarray*}
		Since $\Gal(K_{\O}(p)/\Q)^{\ab}$ corresponds to the maximal abelian extension of $\Q$ that is unramified over $K$ (this is called the genus field of $K$), the inequality in the theorem immediately follows. 
		
		In the rest of the proof, we will construct infinitely many fields for which the equality in the theorem holds. 
		
		{\bf Case I: $p$ is odd. }We first consider the case when $p$ is odd to avoid the complexity of ramification at the archimedean place. We will construct a set of prime numbers $\{\frakp_1, \ldots, \frakp_n\}$ by induction. First, pick a prime $\frakp_1$ of $\Q$ such that $\Nm(\frakp) \equiv 1 \mod |I_1|$. Assume we have $\frakp_1, \ldots, \frakp_i$ for some $1\leq i <n$. We pick a prime $\frakp_{i+1}$ satisfying the followings
		\begin{enumerate}[label=(\alph*)]
					\item\label{item:pcond-1} $\frakp_{i+1}$ splits completely in $\Q(\mu_{|I_{i+1}|}, \sqrt[p]{\frakp_1}, \ldots, \sqrt[p]{\frakp_i})/\Q$,
				\item\label{item:pcond-2} $\frakp_{i+1}$ splits completely in $E_p(\frakp_j)/\Q$ for any $2\leq j \leq i$, and  
			\item\label{item:pcond-3} $\frakp_{i+1}$ is completely inert in $E_{p}(\frakp_1)/\Q$ (i.e., the Frobenius element at $\frakp_{i+1}$ generates the whole Galois group of $E_{p}(\frakp_1)/\Q$).
		\end{enumerate}
		Note that $\frakp_{i+1}$ splits completely in $\Q(\mu_{|I_{i+1}|})/\Q$ if and only if $\Nm(\frakp_{i+1})\equiv 1 \mod |I_{i+1}|$, so the requirement \ref{item:pcond-1} implies that $E_{|I_{i+1}|}(\frakp_{i+1})$ (and hence $E_p(\frakp_{i+1})$) exists. Because the extensions in \ref{item:pcond-1}, \ref{item:pcond-2} and \ref{item:pcond-3} are pairwise disjoint, it follows by Chebotarev's density theorem that there exist infinitely many primes satisfying all the requirements above. 
		
		So we have a set $\{\frakp_1, \ldots, \frakp_n\}$. There exists a subfield $K$ of the composition field $\prod_{i=1}^n E_{|I_i|}(\frakp_i)$ defined by the following quotient map (i.e., $K$ is the subfield fixed by the kernel of the following map)
		\begin{equation}\label{eq:const-K}
			\Gal \left(\prod_{i=1}^n E_{|I_i|}(\frakp_i) /\Q \right) = \prod_{i=1}^n \Gal(E_{|I_i|}(\frakp_i)/ \Q) = \prod_{i=1}^n I_i \longrightarrow \Gamma,
		\end{equation}
		where the last arrow is defined by the embeddings of subgroup $I_i$ into $\Gamma$ for all $i$. One can check by our construction that $K$ is a field in $\calS(\Gamma, n, \{I_i\}_{i=1}^n, I_{\infty})$. Note that $K$ is totally real as $p$ is odd. So we have a quotient map
		\[
			\Conv_{i=1}^n I_i \overset{\phi}{\longrightarrow} \Gal(K_{\O}(p)/\Q) \longrightarrow \Gal(K/\Q),
		\]
		where the first arrow is the map $\phi$ in Corollary~\ref{cor:Q} and the second is the natural projection between Galois groups. 
		Let $y_i$ be as defined in Corollary~\ref{cor:Q}.
		Recall that we defined a quotient map $\psi$ at the beginning of this subsection. By Lemma~\ref{lem:recip} and the requirements \ref{item:pcond-1} and \ref{item:pcond-2}, for each $2\leq i \neq j \leq n$, we have $\frakp_i$ splits completely in $E_p(\frakp_j)$. This together with the requirement \ref{item:pcond-3} implies that 
		\[
			\langle \psi(y_i) \rangle = \langle \psi(x_1) \rangle \text{ for each } i \geq 2
		\] 
		as subgroups of $\im \psi$. So by Lemma~\ref{lem:lb-cyclic}, the images of $x_i^{\Nm(\frakp_i)} y_i x_i^{-1} y_i^{-1}$ for $2 \leq i \leq n$ under the map $\alpha$ generate the commutator subgroup of $\im \alpha$. Note that the image of $x_1^{\Nm(\frakp_1)} y_1 x_1^{-1} y_1^{-1}$ is obviously contained in the commutator subgroup of $\im \alpha$.
		Then, by Lemma~\ref{lem:cyc-structure} and Corollary~\ref{cor:Q}, we have $\dim_{\F_p}\Cl(K)[p]=n-1$. Finally, since there are infinitely many choices of the set $\{\frakp_1, \ldots, \frakp_n\}$, there are infinitely many $K$ for which the equality in the theorem holds. 
		
		{\bf Case II.1: $p=2$ and $|I_{\infty}|=2$.} Note that, if $2^m \parallel \Nm(\frakp)-1$, then $E_{2^m}(\frakp)$ is imaginary and $E_{2^{m-1}}(\frakp)$ is real. We just repeat the argument in Case I, and when choosing $\{\frakp_1, \ldots, \frakp_n\}$, we further require
		\[
			|I_1| \parallel \Nm(\frakp_1)-1 \quad \text{and} \quad |I_i| \not\parallel \Nm(\frakp_i)-1 \text{ for $i\geq 2$}.
		\]
		(These conditions can be rewritten as ``the Frobenius element at $\frakp_i$ in $\Gal(\Q(\mu_{2|I_i|})/\Q)$ generate a particular subgroup'', so the primes $\frakp_i$'s exist by Chebotarev's density theorem.)
		So $E_{|I_1|}(\frakp_1)$ is imaginary and $E_{|I_i|}(\frakp_i)$ is real for $i \geq 2$. Then the field $K$ defined by \eqref{eq:const-K} is imaginary, so is in $\calS(\Gamma, n, \{I_i\}_{i=1}^n, I_{\infty})$. Similarly to Case I, the images of $x_i^{\Nm(\frakp_i)} y_i x_i^{-1} y_i^{-1}$ for $1\leq i \leq n$ under $\alpha$ generate the commutator subgroup of $\im \alpha$. Since $x_{\infty}^2$ is contained in the Frattini subgroup of $\Conv_{i=1}^n I_i$, so modulo $x_{\infty}^2$ does not reduce the generator rank. Then the generator rank of $\Gal(K_{\O}(2)/\Q)$ is $n$, so $\dim_{\F_2} \Cl(K)[2]=n-1$.
		
		{\bf Case II.2: $p=2$ and $I_{\infty}=1$.} When $n=1$, the theorem is obvious. So we assume $n\geq 2$. When choosing $\{\frakp_1, \ldots, \frakp_n\}$, we additionally require
		\[
			|I_i| \parallel \Nm(\frakp_i)-1 \text{ for $i=1$ and $n$}, \quad \text{and} \quad |I_i| \not \parallel \Nm(\frakp_i)-1 \text{ for other $i$.}
		\]
		Then one can check that $K$ is totally real, and hence is in $\calS(\Gamma, n, \{I_i\}_{i=1}^n ,1)$. In the abelianization 
		\[
			I_1 \times I_2 \times \cdots \times I_{n-1} \times I_n
		\]
		of $\Conv_{i=1}^n I_i$, the image of $x_{\infty}$ generate the subgroup $C_2 \times 1 \times \cdots \times 1 \times C_2$. So modulo-$x_{\infty}$ reduces the generator rank if and only if $I_n$ is $C_2$. So we see that $\dim_{\F_2}\Cl(K)[2]$ achieves the lower bound given in the theorem.
	\end{proof}
	
\subsection{Upper bound}
	For a number field $k$ and $\frakp$ a prime of $\Q$, we denote
	\[
		k_{\{\frakp\}}^{\el} := \,\begin{aligned}&\text{the maximal elementary abelian $p$-extension} \\ &\text{of $k$ that is ramified at only primes above $\frakp$}. \end{aligned}
	\]

	\begin{lemma}\label{lem:for-wreath}
		Let $k/\Q$ be Galois such that $\Cl(k)[p]=1$, and let $e_1, \ldots, e_r$ denote a basis of $\calO_k^{\times}$. If $\frakp$ splits completely in $k(\mu_p, \sqrt[p]{e_1}, \ldots, \sqrt[p]{e_r})/\Q$, then there is an isomorphism
		\[
			\Gal(k^{\el}_{\{\frakp\}}/\Q) \simeq \F_p \wr \Gal(k/\Q)
		\]
		that is compatible with respect to the natural surjections from both sides to $\Gal(k/\Q)$.
	\end{lemma}
	
	\begin{proof}
		Recall the group
		\[
			V_S(k):= \ker \left( k^{\times}/k^{\times p} \longrightarrow \prod_{\frakP \in S} k_{\frakP}^{\times}/k_{\frakP}^{\times p} \times \prod_{\frakP \not\in S} k_{\frakP}^{\times} / U_{\frakP} k_{\frakP}^{\times p} \right)
		\]
		which is $V_{S}^S(k)$ in \eqref{eq:def-V} and is the Pontryagin dual of $\B_S(k):=\B_S^S(k)$.
		It's clear that $V_{\{\frakp\}}(k)$ naturally embeds into $V_{\O}(k)$. By \cite[Cor.(10.7.2)]{NSW}, the assumption $\Cl(k)[p]=1$ implies that the embedding $\calO_{k}^{\times} \hookrightarrow k^{\times}$ induces an isomorphism
		\[
			\calO_k^{\times}/\calO_k^{\times p} \overset{\sim}{\longrightarrow} V_{\O}(k).
		\]
		So,
		\[
			V_{\{\frakp\}}(k) \simeq \ker \left( \calO_k^{\times} / \calO_k^{\times p} \longrightarrow \prod_{\frakP \mid \frakp} k_{\frakP}^{\times}/k_{\frakP}^{\times p} \right).
		\]
		The assumption that $\frakp$ splits completely in $k(\mu_p, \sqrt[p]{e_1}, \ldots, \sqrt[p]{e_r})/\Q$ implies that the image of the local completion map $\calO^{\times}_k \hookrightarrow k_{\frakP}^{\times}$ is contained in $k_{\frakP}^{\times p}$ for every $\frakP$ of $k$ lying above $\frakp$. So
		\[
			\calO_k^{\times}/ \calO_k^{\times p} \longrightarrow k_{\frakP}^{\times} / k_{\frakP}^{\times p}
		\]
		is the zero map. Therefore, we have $V_{\O}(k)=V_{\{\frakp\}}(k)$, and hence $\B_{\O}(k)=\B_{\{\frakp\}}(k)$.
		
		Then, by \cite[Lem.(10.7.4)(i)]{NSW}, there is a short exact sequence
		\[
			0 \longrightarrow H^1(G_{\O}(k), \F_p) \longrightarrow H^1(G_{\{ \frakp\}}(k), \F_p) \longrightarrow \bigoplus_{\frakP \mid \frakp} H^1(\calT_{\frakP}, \F_p)^{\calG_{\frakP}} \longrightarrow 0,
		\]
		and moreover, this short exact sequence is equivariant with respect to the conjugation action by $\Gal(k/\Q)$ (one can check it from the proof of \cite[Lem.(10.7.4)(i)]{NSW}). Since $\Cl(k)[p]=1$, the $p$-part of $\Cl(k)$ is trivial, and hence $H^1(G_{\O}(k), \F_p)$ is trivial. The assumption that $\frakp$ splits completely in $k(\mu_p)/\Q$ also implies that $\frakp \neq p$, $k_{\frakP}=\Q_{\frakp}$ for every $\frakP \mid \frakp$, and $\Nm(\frakp)\equiv 1 \mod p$. So $H^1(\calT_{\frakP}, \F_p) \simeq \F_p$ with trivial $\calG_{\frakP}$-actions, and thus we have the following isomorphism of $\Gal(k/\Q)$-modules
		\[
			\bigoplus_{\frakP \mid \frakp} H^1(\calT_{\frakP}, \F_p)^{\calG_{\frakP}}\simeq \F_p[\Gal(k/\Q)].
		\]
		So
		\[
			\Gal(k^{\el}_{\{\frakp\}}/k) = H^1(G_{\{ \frakp\}}(k), \F_p)^{\vee} \simeq \F_p[\Gal(k/\Q)].
		\]
		Finally, as $H^2(\Gal(k/\Q), \F_p[\Gal(k/\Q)])=0$ by the Shapiro's lemma, it follows by the bijection between $H^2$ and $\Ext$ (\cite[Thm.(1.2.4)]{NSW})
		that the following extension is split
		\[
			1 \longrightarrow \Gal(k^{\el}_{\{\frakp\}}/k) \longrightarrow \Gal(k^{\el}_{\{\frakp\}}/\Q) \longrightarrow \Gal(k/\Q) \longrightarrow 1.
		\]
		So the lemma follows.
	\end{proof}
	
	\begin{proposition}\label{prop:ub-cyclic}
		Let $\Gamma$ be a cyclic $p$-group. For any ramification type $(\Gamma, n, \{I_i\}_{i=1}^n, I_{\infty})$, there are infinitely many $K \in \calS(\Gamma, n, \{I_i\}_{i=1}^n, I_{\infty})$ such that 
		\[
			\dim_{\F_p} \Cl(K)[p] \geq \sum_{i=1}^n \frac{|\Gamma|}{|I_i|} -1.
		\]
	\end{proposition}
	
	\begin{proof}
		We assume $I_1=\Gamma$. We first pick a prime $\frakp_1$ satisfying $|I_1| \mid \Nm(\frakp_1)-1$, and furthermore satisfying $|I_1| \parallel \Nm(\frakp_1)-1$ when $p=2$ and $|I_\infty|=2$, and $|I_1| \not\parallel \Nm(\frakp_1)-1$ when $p=2$ and $I_\infty=1$. We set $k:=E_{|I_1|}(\frakp_1)$. Then $k$ is real if $I_{\infty}=1$ and imaginary if $|I_{\infty}|=2$, and $\Cl(k)[p]=1$ by Corollary~\ref{cor:Q-rank}.
		By Lemma~\ref{lem:for-wreath}, we can pick distinct primes $\frakp_2, \ldots, \frakp_n$ such that $|I_i| \mid \Nm(\frakp_i)-1$ and
		\[
			\Gal(k_{\{ \frakp_i\}}^{\el}/\Q) \simeq \F_p \wr \Gal(k/\Q), \quad \text{for any $i\geq 2$}.
		\]
		When $p=2$, we also require each $\frakp_i$ for $i\geq 2$ to split completely in $\Q(\mu_{2|I_i|})$ so that $E_{|I_i|}$ is real.
		
		We consider the field
		\[
			L_1 := \prod_{i=2}^n k_{\frakp_i}^{\el},
		\]
		that is Galois over $\Q$ with $\Gal(L_1/\Q) \simeq \F_p[\Gamma]^{\oplus n-1} \rtimes \Gamma$ (this isomorphism is compatible with quotient maps to $\Gal(k/\Q)\simeq \Gamma$). For each $i$, there is a natural quotient map $\F_p[\Gamma] \to \F_p[\Gamma/I_i]$ of modules, defined by mapping each $g \in I_i$ to 0. Then the quotient map
		\[
			\F_p[\Gamma]^{\oplus n-1} \rtimes \Gamma \longrightarrow \left(\bigoplus_{i=2}^n \F_p[\Gamma/I_i] \right) \rtimes \Gamma
		\]
		defines a subfield $L_2$ of $L_1$. 
		
		We define 
		\[
			L_3:=\prod_{i=1}^n E_{|I_i|}(\frakp_i) \quad \text{and} \quad L:=L_2L_3.
		\]
		Note that the intersection $L_2 \cap L_3$ is $k \prod_{i=2}^n E_p(\frakp_i)$. So we have the following diagram describing a fiber product relation of all these Galois groups
		\begin{equation}\label{eq:field-fiber}
		\begin{tikzcd}
			\Gal(L/\Q) \arrow[two heads]{r} \arrow[two heads]{d} & \Gal(L_2/\Q) \simeq \left(\bigoplus_{i=2}^n \F_p[\Gamma/I_i] \right) \rtimes \Gamma \arrow[two heads]{d} \\
			\Gal(L_3/\Q) \simeq \prod_{i=1}^n I_i \arrow[two heads]{r} & \Gal(L_2 \cap L_3 /\Q) \simeq I_1 \times  (C_p)^{\oplus n-1}.
		\end{tikzcd}
		\end{equation}
		Because $L_2/L_2 \cap L_3$ is unramified, we have $L/L_3$ is also unramfied.
		The map \eqref{eq:const-K} defines a subfield $K$ of $L_3$ such that $K \in \calS(\Gamma, n , \{I_i \}_{i=1}^n, I_{\infty})$ and $L_3/K$ is unramified. So $L/K$ is unramified.
	
		In the rest, we will prove that the maximal elementary abelian $p$-extension of $K$ contained in $L$ has Galois group $(C_p)^{\oplus (\sum_{i=2}^n |\Gamma|/|I_i|)}$ over $K$, and then the proposition immediately follows. Consider the quotient maps between Galois groups
		\[
			\Gal(L/\Q) \overset{\varpi_1}{\longrightarrow} \Gal(L_3/\Q) \overset{\varpi_2}{\longrightarrow} \Gal(K/\Q),
		\]
		and denote $\varpi:= \varpi_2 \circ \varpi_1$.
		Then we obtain the short exact sequence
		\begin{equation}\label{eq:ses-ker}
			1 \longrightarrow \ker \varpi_1 \longrightarrow \ker \varpi \longrightarrow \ker \varpi_2 \longrightarrow 1.
		\end{equation}
		Let $x_i$ denote a generator of $\calT_{\frakp_i}(L/\Q)$ in $\Gal(L/\Q)$. Then, from the diagram \eqref{eq:field-fiber}, we can describe $\ker \varpi_1$, $\ker \varpi_2$, and $\ker \varpi$ as follows
		\begin{eqnarray*}
			\ker \varpi_1 &=& \bigoplus_{i=2}^n\, \left[ x_i^{-1}  x_i^{x_1} \right]_{\Gal(L/\Q)} \simeq \bigoplus_{i=2}^n \,\ker(\F_p[\Gamma/I_i] \to \F_p)\\
			\ker \varpi_2 &=& \bigoplus_{i=2}^n\, \left\langle \varpi_1(x_i^{-1} x_1^{m_i})\right\rangle \simeq \bigoplus_{i=2}^n I_i, \\
			\ker \varpi &=& \bigoplus_{i=2}^n \left[ x_i^{-1} x_1^{m_i} \right]_{\Gal(L/\Q)},
		\end{eqnarray*}
		where the symbol $[-]_{\Gal(L/\Q)}$ denotes the normal subgroup of $\Gal(L/\Q)$ generated by the element inside the bracket, and $m_i$ is an interger such that $\varpi(x_i)=\varpi(x_1)^{m_i}$ (such $m_i$ exists because $\varpi(x_1)$ generates $\Gal(K/\Q)\simeq \Gamma$). Because $x_i^{m_i} \in I_i$ acts trivially on $x_i$, we see that
		\[
			(x_i^{-1} x_1^{m_i})^{|I_i|}=x_i^{-|I_i|} x_1^{m_i|I_i|}=1.
		\]
		So the order of $x_i^{-1} x_1^{m_i}$ equals the order of its image under $\varpi_1$, and then the $i$-th component of groups in \eqref{eq:ses-ker}
		\[
			1 \longrightarrow \left[x_i^{-1} x_i^{x_1}\right]_{\Gal(L/\Q)} \longrightarrow \left[x_i^{-1} x_1^{m_i}\right]_{\Gal(L/\Q)} \longrightarrow \left\langle \varpi_1(x_i^{-1} x_1^{m_i}) \right\rangle \longrightarrow 1
		\]
		gives a split group extension.
		Also, since $I_i$ acts trivially on $\ker(\F_p[\Gamma/I_i] \to \F_p)$ and the image of $\varpi_1(x_i^{-1}x_1^{m_i})$ in $\Gal(k/\Q)\simeq \Gamma$ belongs to $I_i$, the extension above is central. Therefore, 
		\[
			\left[x_i^{-1} x_1^{m_i}\right]_{\Gal(L/\Q)} = \left[x_i^{-1} x_i^{x_1}\right]_{\Gal(L/\Q)}  \times \left\langle \varpi_1(x_i^{-1} x_1^{m_i}) \right\rangle \simeq \ker(\F_p[\Gamma/I_i] \to \F_p) \times I_i.
		\]
		So $\ker \varpi$ is an abelian $p$-group, whose Frattini quotient is isomorphic to $(C_p)^{\oplus (\sum_{i=2}^n |\Gamma|/|I_i|)}$.
	\end{proof}

\section{2-rank of class groups of multiquadratic fields}\label{S:multiquad}

\subsection{Descending central filtration}\label{SS:filtration}
	Consider the surjection $\pi$ in \eqref{eq:const-pi} for a ramification type $(\Gamma, n, \{I_i\}_{i=1}^n, I_{\infty})$ and $\Gamma=(C_2)^{\oplus d}$.
	Then by taking quotient of $\Conv_{i=1}^n C_2$ modulo $\Phi(\ker \pi)$, the map $\pi$ induces the following short exact sequence
	\begin{equation}\label{eq:ses}
		1 \longrightarrow N \longrightarrow F \longrightarrow \Gamma \longrightarrow 1,
	\end{equation}
	where $N$ is $\ker \pi/\Phi(\ker\pi)$ and $F$ is $(\Conv_{i=1}^n C_2) / \Phi(\ker\pi)$. By the Kurosh subgroup theorem, $\ker \pi$ is the free profinite group generated of rank $(n-2)2^{d-1}+1$, so 
	\[
		N \simeq \F_2^{\oplus (n-2)2^{d-1}+1}
	\]
	as groups. The conjugation action of $F$ on $N$ defines a $\F_2[\Gamma]$-module structure on $N$. We will study the module structure by the filtration of $F$ given by the descending central series:
	\[
		F_{(1)}:= F \quad \text{and} \quad F_{(i+1)}:=[F_{(i)}, F].
	\]
	The quotient $\gr_i:=F_{(i)}/F_{(i+1)}$ is abelian, which we write additively, and we define
	\[
		\gr:=\bigoplus_{i=1}^{\infty} \gr_i.
	\]
	Note that $\im \pi$ is abelian, and it follows that $F_{(2)}\subseteq N$. So each $\gr_i$ has exponent 2, and then $\gr$ has the structure of a Lie algebra over $\F_2$ with the Lie bracket induced by the commutator (that is, if $x\in F_{(i)}$ and $y\in F_{(j)}$, then $[\overline{x},\overline{y}]$ is the image of $[x,y]$ in $\gr_{i+j}$, where $\overline{x}$ and $\overline{y}$ are the images of $x$ and $y$ in $\gr_i$ and $\gr_j$ respectively.
	
	\begin{lemma}\label{lem:gr}
		Let $x \in F_{(i)}$, $y \in F_{(j)}$ and $z \in F_{(k)}$. Then the followings hold.
		\begin{enumerate}
			\item\label{item:gr-1} 
				$[x,y]=[y,x]$.
			\item\label{item:gr-2}
				$[xy,z]\equiv[x,z][y,z] \mod F_{(i+j+k)}$.
			\item\label{item:gr-3}
				If $x,y \in N$, then $[x,y]=1$.
			\item\label{item:gr-4} If $z\in N$, then $[x,[y,z]]= [y,[x,z]]$.
		
			\item\label{item:gr-5} If $x_1, \ldots, x_m$ are elements of $F$, then 
			\[
				[x_1,[x_2,[\ldots, [x_{m-1}, x_m]]\ldots]\,\cdot\,[x_2,[x_3,[\ldots, [x_{m}, x_1]]\ldots] \,\cdots 
			\]
			\[
				\cdot \,[x_m,[x_1,[\ldots, [x_{m-2}, x_{m-1}]]\ldots] = 1 .
			\]
		\end{enumerate}
	\end{lemma}
	
	\begin{proof}
		Note that $[x,y]=xyx^{-1}y^{-1}= [y, x]^{-1}$ and that $[F,F]\subseteq N$ has exponent 2, so $[x,y]=[y,x]$ and the statement~\eqref{item:gr-1} follows. There is an identity of commutators
		\begin{equation}\label{eq:com-prod}
			[xy,z]=[x,[y,z]] \cdot [y,z]\cdot [x,z].
		\end{equation}
		Since $[x,[y,z]]$ is in $F_{(i+j+k)}$, the statement \eqref{item:gr-2} follows. The statement \eqref{item:gr-3} is because $N$ is abelian.
		
		We apply an identity of commutators
		\begin{equation}\label{eq:id}
			[x^z, [y,z]] \, \cdot \, [z^y, [x,y]] \,\cdot\, [y^x, [z,x]] =1.
		\end{equation}
		As $x^z=[z,x]x$, it follows by the statement \eqref{item:gr-3} and the identity \eqref{eq:com-prod} that $[x^z,[y,z]]= [x,[y,z]]$. Similarly, we have $[y^x, [z,x]]= [y,[z,x]]$ and $[z^y,[x,y]]= [z,[x,y]]=1$ when $z \in N$. So the identity \eqref{eq:id} implies that $[x,[y,z]] \, \cdot\, [y,[z,x]]=1$. Recalling that $N$ has exponent 2, the statement \eqref{item:gr-4} immediately follows by applying \eqref{item:gr-1}. 
		
		Finally, we prove \eqref{item:gr-5} by induction. It holds for $m=2$ by the statement~\eqref{item:gr-1}. Assume that it holds for $m=l-1$. We apply the following identity of commutators 
		\begin{equation}\label{eq:id-1}
			[x_1, x_2 x_3 \cdots x_l]\,\cdot \, [x_2,x_3x_4\cdots x_l x_1] \,\cdots \, [x_l, x_1x_2\cdots x_{l-1}]=1.
		\end{equation}
		By the statement~\eqref{item:gr-1} and applying \eqref{eq:com-prod} inductively, because $N$ is abelian, we have 
		\[
			[x_1,x_2 x_3 \cdots x_l]=\prod_{i=1}^{l-1}\,\prod_{2\leq s_1 < s_2 < \ldots < s_i \leq l} [x_{s_1}, [x_{s_2},[\ldots, [x_{s_i}, x_1]]\ldots],
		\]
		and similarly, we have such a formula for each term in the product \eqref{eq:id-1}. Then \eqref{eq:id-1} implies
		\[
			\prod_{i=2}^l\, \prod_{1\leq s_1< \ldots <s_i\leq l} \bigg([x_{s_1}, [x_{s_2},[\ldots, [x_{s_{i-1}}, x_{s_i}]]\ldots] \,\cdot\, [x_{s_2}, [x_{s_3},[\ldots, [x_{s_{i}}, x_1]]\ldots]\, \cdots
		\]
		\[
			\cdot \, [x_{s_i}, [x_{s_1},[\ldots, [x_{s_{i-2}}, x_{s_{i-1}}]]\ldots] \bigg) =1.
		\]
		The product of terms with $i<l$ is 1 because of our inductive hypothesis, and hence the identity in \eqref{item:gr-5} with $m=l$ follows.
	\end{proof}
	
	From now on, we let $x_i$, $i=1, \ldots, n$ denote the element of $F$ that is the image of the nontrivial element of the $i$-th copy of $C_2$ in $\Conv_{i=1}^n C_2$. 
	
	\begin{lemma}\label{lem:repeat}
		Let $s_1, \ldots, s_m$ be integers in $\{1, \ldots, n\}$. Then if $s_j=s_k$ for some distinct $j$ and $k$, then 
		\[
			[x_{s_1}, [x_{s_2}, [\ldots, [x_{s_{m-1}}, x_{s_m}]]\ldots ] =1.
		\]
	\end{lemma}
	
	\begin{proof}
		If $s_{m-1}=s_m$, then $[x_{s_{m-1}}, x_{s_m}]=1$ and the lemma is obvious. If $s_{m-1}\neq s_m$, then we can apply Lemma~\ref{lem:gr} \eqref{item:gr-1} and \eqref{item:gr-4} appropriately to rearrange the ordering of the elements in the commutator so that $x_{s_j}$ and $x_{s_k}$ are next to each other. In other words, we can rewrite the commutator in the lemma so that at some position $[x_{s_j}, [x_{s_k}, c]]$ appears as an inner commutator for some $c \in F$. By the identity~\eqref{eq:com-prod}, it follows $[x_{s_j}, [x_{s_k}, c]]=[x_{s_j}x_{s_k}, c]\,\cdot\,[x_{s_j},c]\,\cdot\,[x_{s_k},c]=1$, where the last step uses that both $x_{s_j}=x_{s_k}$ and $[x_{s_j},c]=[x_{s_k},c]$ have order 2. Therefore, we proved the lemma.
	\end{proof}

	\begin{lemma}\label{lem:basis}
		Assume $n=d$. Then
		\begin{enumerate}
			\item\label{item:basis-1}
			\[
				\dim_{\F_2}(\gr_i)=\begin{cases}
					d & \text{if } i=1 \\
					(i-1) \cdot \binom{d}{i} & \text{if } 2 \leq i \leq d \\
					0 & \text{if } i >d.
				\end{cases}
			\]
			
			\item\label{item:basis-2} The images of $x_1, x_2, \ldots, x_d$ form a basis of $\gr_1$. For $2\leq i \leq d$, the images of elements of the following form give a basis of $\gr_i$:
			\begin{eqnarray}
				&[x_{s_1}, [x_{s_2}, [\ldots, [x_{s_{i-1}},x_{s_i}]]\ldots], & \nonumber \\
				&[x_{s_2}, [x_{s_3}, [\ldots, [x_{s_{i}},x_{s_1}]]\ldots], & \nonumber\\
				&\vdots & \nonumber\\
				&[x_{s_{i-1}}, [x_{s_i}, [\ldots, [x_{s_{i-3}},x_{s_{i-2}}]]\ldots] & \label{eq:gen}
			\end{eqnarray}
			with $1\leq s_1<s_2 <\ldots <s_i\leq d$.
			
		\end{enumerate}
	\end{lemma}
	
	\begin{proof}
		We first prove the statement~\eqref{item:basis-2}. Note that $x_1,\ldots, x_d$ form a generating set of $F$, so every element of $F$ can be written as a word from this set. The first statement in \eqref{item:basis-2} is obvious. Then by applying Lemma~\ref{lem:gr} \eqref{item:gr-1} and \eqref{item:gr-2} inductively, it follows that $\gr_i$ is generated by images of $[x_{t_1}, [x_{t_2}, [\ldots, [x_{t_{i-1}}, x_{t_i}]]\ldots ]$ with $t_1, \ldots, t_i \in \{1, \ldots, d\}$.
		By Lemma~\ref{lem:repeat}, $\gr_i$ is generated by images of $[x_{t_1}, [x_{t_2}, [\ldots, [x_{t_{i-1}}, x_{t_i}]]\ldots ]$ with $\{t_1, \ldots, t_i\} \subseteq \{1, \ldots, d\}$, and $\gr_i=0$ for $i>d$. 
		
		Next, we will prove that, for any $\{t_1, \ldots, t_i\} \subseteq \{1, \ldots, d\}$, the element $[x_{t_1}, [x_{t_2}, [\ldots, [x_{t_{i-1}}, x_{t_i}]]\ldots ]$ with $\{t_1, \ldots, t_i\} \subseteq \{1, \ldots, d\}$ can be generated by the elements as described in \eqref{eq:gen} for the index set $\{s_1, \ldots, s_i\}=\{t_1, \ldots, t_i\}$. Without loss of generality, we assume $\{t_1, \ldots, t_i\}=\{1, \ldots, i\}$ and $t_i=1$, and we write $k:=t_{i-1}$. We let $V$ denote the subspace of $\gr_i$ generated by \eqref{eq:gen} for $s_j=j$ with $j=1, \ldots, i$. By Lemma~\ref{lem:gr} \eqref{item:gr-4}, it suffices to show that 
		\begin{equation}\label{eq:ind}
			\text{the image of } y_{k}:=[x_2, [x_3, [\ldots, [x_{k-1},[x_{k+1}, [ x_{k+2}, [\ldots, [x_i, [x_k, x_1]]\ldots ]  \text{ is contained in $V$.}
		\end{equation}
		We prove it by induction. When $k=2$, we have $y_2 = [x_3, [x_4, [\ldots,[x_1, x_2]]\ldots ]$ by Lemma~\ref{lem:gr} \eqref{item:gr-1}, and therefore \eqref{eq:ind} holds for $k=2$. Assume \eqref{eq:ind} holds for $k=l$ for some $l<i$. Again, by swapping the contiguous elements, we have
		\[
			y_l= [x_2, [\ldots, [x_{l},[x_{l+2}, [\ldots, [x_i, [x_{l+1}, [x_1, x_l]]\ldots ],
		\]
		and similarly, we have the following identity
		\[
			z:=[x_{l+2},  [\ldots, [x_{l-1}, [x_l, x_{l+1}]]\ldots ] = [x_2, [\ldots, [x_{l},[x_{l+2},  [\ldots, [x_i, [x_1, [x_l, x_{l+1}]]\ldots ],
		\]
		where by a slight abuse of notation we write $x_{m}:=x_{(m-1\mod i)+1}$ for any integer $m$. By Lemma~\ref{lem:gr} \eqref{item:gr-5}, we have $[x_{l+1}, [x_1, x_l]] \,\cdot\, [x_1, [x_l, x_{l+1}]]=[x_l,[x_{l+1}, x_1]]$. Thus, by Lemma~\ref{lem:gr} \eqref{item:gr-2}, we have $y_{l+1} \equiv z \cdot y_l \mod F_{(i+1)}$. Then it follows by $z, y_l \in V$ that $y_{l+1} \in V$, and hence we proved the statement in \eqref{eq:ind}.
		
		Now, we have shown that images of the elements as described in \eqref{eq:gen} generate $\gr_i$. So we have
		\begin{equation}\label{eq:up}
			\dim_{\F_2} (\gr_i) \leq 
				(i-1) \cdot \binom{d}{i}  \text{ for } 2\leq i \leq d,
					\end{equation}
		and $\gr_i=0$ for $i>d$. On the other hand, recall that by the Kurosh subgroup theorem, we have $\dim_{\F_2}(N)=(d-2) 2^{d-1}+1$. So it follows by the fact $\dim_{\F_2} N = \sum_{i=2}^d \dim_{\F_2} \gr_i$ and the binomial coefficient formulas
		\[
			\sum_{i=1}^d i \binom{d}{i}=d 2^{d-1} \quad \text{and} \quad \sum_{i=1}^d \binom{d}{i}=2^d-1
		\]
		that the equality in \eqref{eq:up} holds for all $i$, and therefore we finish the proof.
	\end{proof}

	\begin{lemma}\label{lem:basis+}
		Assume $n>d$. 
		\begin{enumerate}
			\item\label{item:basis+-1}
				\[
					\dim_{\F_2}(\gr_i)=\begin{cases}
						n & \text{if }i=1 \\
						(i-1)\cdot \binom{d}{i}+(n-d) \cdot \binom{d-1}{i-1} &\text{if } 2 \leq i \leq d \\
						0 & \text{if } i>d.
					\end{cases}
				\]
			\item\label{item:basis+-2} Without loss of generality, assume that the images of $x_1, \ldots, x_d$ (under the quotient map in \eqref{eq:ses}) form a basis of $\Gamma$. For each $d<j\leq n$, we choose an integer $1\leq e_j\leq d$ such that the images of $x_1, \ldots, x_{e_j-1}, x_{e_j+1}, \ldots, x_d$ and $x_j$ form a basis of $\Gamma$. Then the images of $x_1, x_2, \ldots, x_n$ form a basis of $\gr_1$. For $2 \leq i \leq d$, the images of the elements of the following two types form a basis of $\gr_i$:
			\begin{description}
				\item[Type I] 
					\begin{eqnarray*}
					&[x_{s_1}, [x_{s_2}, [\ldots, [x_{s_{i-1}},x_{s_i}]]\ldots], &  \\
					&[x_{s_2}, [x_{s_3}, [\ldots, [x_{s_{i}},x_{s_1}]]\ldots], & \\
					&\vdots & \\
					&[x_{s_{i-1}}, [x_{s_i}, [\ldots, [x_{s_{i-3}},x_{s_{i-2}}]]\ldots] & 
				\end{eqnarray*}
				with $1\leq s_1<s_2 <\ldots <s_i\leq d$;
				
				\item[Type II] 
					\[
						[x_{s_1}, [x_{s_2}, [\ldots, [x_{s_{i-1}},x_{j}]]\ldots], 
					\]
					with $j=d+1, \ldots, n$, and $s_1< \ldots<s_{i-1}$ integers in $\{1, \ldots, d\} \backslash \{e_j\}$.
			\end{description}
			
		\end{enumerate}
	\end{lemma}
	
	\begin{proof}
		Denote the abelianization map $\Conv_{i=1}^n C_2 \twoheadrightarrow (C_2)^{\oplus n}$ by $\ab$. By taking the quotient of $\Conv_{i=1}^n C_2$ modulo $\Phi(\ker \ab)$, we obtain a short exact sequence $1 \to \widetilde{N} \to \widetilde{F} \to (C_2)^{\oplus n} \to 1$. We define 
		\[
			\widetilde{\gr}_i:=\widetilde{F}_{(i)}/\widetilde{F}_{(i+1)}, \quad \text{where } \widetilde{F}_{(1)}:=\widetilde{F} \text{ and } \widetilde{F}_{(i+1)}:= [\widetilde{F}_{(i)}, \widetilde{F}].
		\]
		The map $\pi: \Conv_{i=1}^n C_2 \to (C_2)^{\oplus d}$ factors through $\ab$, so $\pi$ and $\ab$ induces a natural surjection 
		\[
			\beta: \widetilde{F} \twoheadrightarrow F.
		\]
		The map $\beta$ preserves the filtration as $\beta(\widetilde{F}_{(i)})=F_{(i)}$. The statement~\eqref{item:basis+-2} implies the formula in the statement~\eqref{item:basis+-1} for $i\leq d$. It follows by the Kurosh subgroup theorem that $\dim_{\F_2}(N)=(n-2)2^{d-1}+1$, which equals to the sum of $\dim_{\F_2}(\gr_i)$ as described in \eqref{item:basis+-1}.
		Note that Lemma~\ref{lem:basis} describes a basis of $\widetilde{\gr}_i$ for each $i$. So it suffices to prove that the elements of Types I and II in \eqref{item:basis+-2} generate the images under $\beta$ of every element in Lemma~\ref{lem:basis}\eqref{item:basis-2}.

		\emph{Claim}: Let $s_1, \ldots, s_m \in \{1, \ldots, n\}$. If the images of $x_{s_1}, \ldots, x_{s_m}$ in $(C_2)^{\oplus d}$ under the quotient map in \eqref{eq:ses} are linear dependent, then $[x_{s_1}, [x_{s_2}, [\ldots, [x_{s_{m-1}}, x_{s_m}]]\ldots ] \equiv 1$ mod $F_{(m+1)}$.
		
		\emph{Proof of Claim}: If $s_{m-1}=s_m$, then the claim obviously holds. Otherwise, because the images of $x_{s_1}, \ldots, x_{s_m}$ are linearly independent, there must be some $k < m-1$ such that 
		\[
			x_{s_k} =x_{t_1}x_{t_2} \cdots x_{t_j} \Delta
		\]
		for some pairwise distinct indices $t_1, \ldots, t_j \in \{1, \ldots, m\} \backslash \{k\}$ and some element $\Delta \in N$. Then by Lemma~\ref{lem:gr}\eqref{item:gr-2},
		\begin{eqnarray}
			&& [x_{s_1}, [x_{s_2}, [\ldots, [x_{s_{m-1}}, x_{s_m}]]\ldots ] \label{eq:comm}\\
			&=& [x_{s_1}, [\ldots, [x_{s_{k-1}}, [\prod_{l=1}^j x_{s_{t_l}} \cdot \Delta, [ x_{s_{k+1}},[ \ldots,[x_{s_{m-1}}, x_{s_m}]]\ldots ] = 1 \nonumber\\
			&\equiv&\prod_{l=1}^j \, [x_{s_1}, [\ldots, [x_{s_{k-1}}, [ x_{s_{t_l}}, [ x_{s_{k+1}},[ \ldots,[x_{s_{m-1}}, x_{s_m}]]\ldots ] \label{eq:comm-1}\\
			&&  \cdot \,  [x_{s_1}, [\ldots, [x_{s_{k-1}}, [\Delta, [ x_{s_{k+1}},[ \ldots,[x_{s_{m-1}}, x_{s_m}]]\ldots ] \mod F_{(m+1)}.\label{eq:comm-2}
		\end{eqnarray}
		Each commutator in the product in \eqref{eq:comm-1} is trivial by Lemma~\ref{lem:repeat}, and the commutator in \eqref{eq:comm-2} is trivial by Lemma~\ref{lem:gr}\eqref{item:gr-3}. So we proved the claim.
		
		It immediately follows by the above claim and Lemma~\ref{lem:basis} \eqref{item:basis-1} and \eqref{item:basis-2} that $\dim_{\F_2} (\gr_i)=0$ for $i>d$. For $i\leq d$, $\gr_i$ is generated by 
		\begin{equation}\label{eq:comm-3}
			[x_{s_1}, [x_{s_2}, [\ldots, [x_{s_{i-1}},x_{s_i}]]\ldots] \quad \text{with } s_1, \ldots, s_i \in \{1, \ldots,  n\}.
		\end{equation}
		For each $j \in \{1, \ldots, n\}$, we can write $x_j$ as $\prod_{k=1}^d x_k^{\epsilon_{j,k}} \cdot \Delta_j$ with $\Delta_j \in N$ and $\epsilon_{j,k}=0$ or 1. Thus, by Lemma~\ref{lem:gr}\eqref{item:gr-2}, the commutator in \eqref{eq:comm-3} can be written as a product of length-$i$ commutators of elements from $x_1, \ldots, x_d$ and $\Delta_{d+1}, \ldots, \Delta_n$, up to modulo $F_{(i+1)}$. Note that by Lemma~\ref{lem:gr}\eqref{item:gr-3} a length-$i$ commutator is trivial if 1) the last two positions are from $N$, or 2) at least one of the first $i-2$ positions is from $N$. Therefore, up to modulo $F_{(i+1)}$, the commutators in \eqref{eq:comm-3} can be generated by elements
		\begin{enumerate}[label=(\roman*)]
			\item\label{item:elm-1} $[x_{t_1}, [x_{t_2}, [\ldots, [x_{t_{i-1}},x_{t_i}]]\ldots]$ with $\{t_1, \ldots, t_i\} \subseteq \{1, \ldots, d\}$, and
			\item\label{item:elm-2} $[x_{t_1}, [x_{t_2}, [\ldots, [x_{t_{i-1}},\Delta_j]]\ldots]$ with $\{t_1, \ldots, t_{i-1}\} \subset \{1, \ldots, d\}$ and $j>d$.
		\end{enumerate}
		By the proof in Lemma~\ref{lem:basis}, elements in \ref{item:elm-1} can be generated by elements of Type I in the lemma. If there exists  $k \in \{1, \ldots, i-1\}$ such that $t_k=e_j$, then we can uniquely write $x_{e_j}$ as product of elements from $\{x_1, \ldots, x_d\} \backslash \{x_{e_j}\}$ and a unique $\Delta'_j \in N$. Then, again by Lemma~\ref{lem:gr}\eqref{item:gr-3}, up to modulo $F_{(i+1)}$, each commutator in \ref{item:elm-2} can be generated by $[x_{t_1}, [x_{t_2}, [\ldots, [x_{t_{i-1}}, \Delta_j]]\ldots]$ with $\{t_1, \ldots, t_{i-1}\} \subset \{ 1, \ldots, d, j\} \backslash \{e_j\}$.
		By Lemma~\ref{lem:gr}\eqref{item:gr-4}, we can appropriately swapping the positions of the commutator, and then we can assume $t_1<t_2< \ldots < t_{i-1}$ in \ref{item:elm-2}. 
		Finally, because
		\[
			[x_{t_1}, [x_{t_2}, [\ldots, [x_{t_{i-1}},x_j]]\ldots]\equiv [x_{t_1}, [x_{t_2}, [\ldots, [x_{t_{i-1}},\Delta_j]]\ldots] \cdot \prod_{k=1}^d [x_{t_1}, [x_{t_2}, [\ldots, [x_{t_{i-1}},x_k]]\ldots]^{\epsilon_{j,k}} \mod F_{(i+1)},
		\]
		it follows by Lemma~\ref{lem:repeat} that each commutator described in \ref{item:elm-2} can be generated by commutators in \ref{item:elm-1} and commutators of Type II. 		
	\end{proof}

\subsection{Lower bound}

	\begin{lemma}\label{lem:1gen-fil}
		Use the notation in Lemma~\ref{lem:basis+}\eqref{item:basis+-2}. If $x$ is an element of $N$ such that $x\in F_{(k)}\backslash F_{(k+1)}$, and $M$ is the $\F_2[\Gamma]$-submodule of $N$ generated by $x$, then for each $i > k$, $(M \cap F_{(i)})/(M\cap F_{(i+1)})$ is generated by the images of 
		\[
			[x_{s_1}, [x_{s_2}, [\ldots, [x_{s_{i-k}}, x]]\ldots],
		\]
		with $1\leq s_1 < s_2 < \ldots < s_{i-k}\leq d$.
	\end{lemma}
	
	\begin{proof}
		The $\Gamma$ action on $N$ is defined by conjugation by $F$, so the submodule $M$ is generated by elements of the form
		\[
			(x_1^{\epsilon_1}x_2^{\epsilon_2} \cdots x_d^{\epsilon_d}) x (x_1^{\epsilon_1}x_2^{\epsilon_2} \cdots x_d^{\epsilon_d})^{-1}
		\]
		with $\epsilon_i \in \{0,1\}$. By the identity~\eqref{eq:com-prod}, $M$ is generated by elements (for all $i$) described in the lemma. Then the lemma follows because $x$ is assumed to be of degree $k$ in the graded filtration $\gr$.
	\end{proof}
	
	\begin{theorem}\label{thm:lb-multiquad}
		Use the notation in Definition~\ref{def:not}. For any ramification type and any $K$ in $\calS((C_2)^{\oplus d},n,\{I_i\}_{i=1}^n, I_{\infty})$,
		\[
			 \dim_{\F_2}\Cl(K)[2]\geq n-d+\sum_{i=2}^{d-1}\max\left\{0, (i-1)\cdot \binom{d}{i} +(n-d) \cdot \binom{d-1}{i-1} -n \cdot \binom{d-1}{i-2}- \binom{d}{i-\alpha_{\infty}} \right\},
		\]
		where $\alpha_{\infty}$ is $1$ if $I_{\infty}=1$ and $2$ if $I_{\infty}\simeq C_2$.

		\begin{proof}
			Since $K$ in $\calS((C_2)^{\oplus d}, n, \{I_i\}_{i=1}^n, I_{\infty})$ is assumed to be tamely ramified, the map $\phi$ is identified with $\pi$ in \eqref{eq:const-pi}. So we obtain a short exact sequence \eqref{eq:ses}, and the generators $x_i$ of $F$ are respectively mapped to the generator of $I_i$ in $\Gamma$. We assume the generators satisfy the condition in Lemma~\ref{lem:basis+}\eqref{item:basis+-2}. So by Corollary~\ref{cor:Q}, $\Cl(K)/2\Cl(K)$ is the quotient of $N$ modulo the $\F_2[\Gamma]$-submodule generated by elements $[x_i, y_i]$ for $i=1, \ldots d$, and the element $x_{\infty}$ if $K$ is real and the element $x_{\infty}^2$ if $K$ is imaginary, where $x_{\infty}\in F$ that corresponds to inertia at $\infty$. We let $M_i$ denote the submodule generated by $[x_i,y_i]$, and let $M_{\infty}$ denote the submodule generated by $x_{\infty}$ if $K$ is real and by $x_{\infty}^2$ if $K$ is imaginary. So
			\[
				\Cl(K)/2\Cl(K) \simeq N/ (M_1M_2\cdots M_n M_{\infty}).
			\]
			By Lemma~\ref{lem:repeat} and the identity \eqref{eq:com-prod}, $[x_i, [x_i, y_i]]=1$, so the subgroup $I_i$ of $\Gamma$ acts trivially on $M_i$, and it follows by Lemma~\ref{lem:1gen-fil} that
			\[
				\dim_{\F_2} (M_i \cap F_{(j)}) / (M_i \cap F_{(j+1)}) \leq \binom{d-1}{j-k_i},
			\]
			where $k_i$ is the degree (in the filtration $\gr$) of the generator $[x_i, y_i]$, for each $i=1, \ldots, n$. And similarly, we have
			\[
				\dim_{\F_2} (M_\infty \cap F_{(j)}) / (M_\infty \cap F_{(j+1)}) \leq \binom{d}{j-k_\infty}.
			\]
			Then because $[x_i, y_i] \in F_{(2)}$ and $x_{\infty}^2\in F_{(2)}$, we have that $k_i \geq 2$, and $k_{\infty}\geq 2$ when $K$ is imaginary and $k_{\infty} \geq 1$ when $K$ is real. Recalling the formula for $\dim \gr_i$ in Lemma~\ref{lem:basis+}\eqref{item:basis+-1}, 
			one can check that when $k_i=2$ and $k_{\infty}=\alpha$, for $j\geq \lfloor \frac{d+3}{2}\rfloor$,
			\[
				\dim_{\F_2} \gr_j \leq \sum_{i=1}^n \dim_{\F_2} \frac{M_i \cap F_{(j)}}{M_i\cap F_{(j+1)}} + \dim_{\F_2}\frac{M_{\infty} \cap F_{(j)}}{M_{\infty} \cap F_{(j)}}.
			\]
			Then comparing different choices of $k_i$ and $k_{\infty}$, one can see that $|M_1M_2\cdots M_n M_{\infty}|$ is maximal when $k_i=2$, $k_{\infty}=\alpha_{\infty}$ and all the $M_i$ and $M_{\infty}$ do not intersect each other when that is unnecessary. 
			The lower bound given in the theorem is the minimum dimension of $N/(M_1\cdots M_n M_{\infty})$ when each $k_i$ is $2$ and $k_{\infty}=\alpha_{\infty}$.
			So the inequality in the theorem follows.
		\end{proof}
	
	\end{theorem}

\subsection{Upper bound}
		Corollary~\ref{cor:Q-rank} gives an upper bound for the 2-rank of the 2-torsion subgroup of the narrow class group of multiquadratic extensions of $\Q$. Koymans and Pagano \cite{Koymans-Pagano} proved the same upper bound for the case that $K/\Q$ is totally real and ramified at only primes that are congruent to 1 modulo 4; moreover, they showed that there are infinitely many fields of a wide number of ramification types for which the upper bound is sharp. Their method applies their previous work on \emph{higher genus theory} \cite{Koymans-Pagano-genus} to give an equivalent description of the sharpness of the upper bound in terms of splittings of primes \cite[Thm.1.2]{Koymans-Pagano}, and then uses combinatorial ideas from the work of Smith \cite{Smith} to show that the splitting conditions are satisfied by infinitely many multiquadratic fields. In this subsection, we will reconstruct the equivalence between the sharpness of the upper bound and the splitting conditions, using our presentation result Corollary~\ref{cor:Q}.
		
		Throughout this subsection, we assume $I_{\infty}=1$ and consider the 2-rank of the narrow class groups.
		
		We first prove the following lemma to show that, in order to see if the upper bound is sharp, it suffices to study it for those ramification types with $n=d$.
		
	\begin{lemma}\label{lem:reduction}
		If $n$ is an integer such that some $\calS((C_2)^{\oplus n},n, \{I_i\}_{i=1}^n, 1)$ contains infinitely many fields $K$ satisfying 
		\begin{equation}\label{eq:up-n}
			\dim_{\F_2} \Cl^+(K)[2] =(n-2)2^{n-1}+1,
		\end{equation}
		then for any ramification type $((C_2)^{\oplus d}, n, \{I'_i\}_{i=1}^n, 1)$ with $d\leq n$, there are infinitely many $K'$ in $\calS((C_2)^{\oplus d},n, \{I'_i\}_{i=1}^n, 1)$ such that 
		\begin{equation}\label{eq:up-d}
			\dim_{\F_2} \Cl^+(K')[2] = (n-2)2^{d-1}+1.
		\end{equation}
	\end{lemma}
	
	\begin{proof}
		First, note that for any two lists of subgroups $\{I_i\}_{i=1}^n$ and $\{J_i\}_{i=1}^n$ of $(C_2)^{\oplus n}$ such that $I_i\simeq J_i\simeq C_2$ and $I_1\cdots I_n=J_1\cdots J_n=(C_2)^{\oplus n}$, there exists an isomorphism of $(C_2)^{\oplus n}$ mapping $I_i$ to $J_i$ respectively, so we have $\calS((C_2)^{\oplus n}, n, \{I_i\}_{i=1}^n, 1)=\calS((C_2)^{\oplus n}, n, \{J_i\}_{i=1}^n, 1)$. By class field theory, different fields in $\calS((C_2)^{\oplus n}, n, \{I_i\}_{i=1}^n, 1)$ cannot ramify at exactly the same set of primes. In the rest of the proof, we will show that for each $K \in \calS((C_2)^{\oplus n}, n, \{I_i\}_{i=1}^n, 1)$ satisfying \eqref{eq:up-n}, there exists $K' \in \calS((C_2)^{\oplus d}, n, \{I'_i\}_{i=1}^n, 1)$ satisfying \eqref{eq:up-d} such that $K'$ is a subfield of $K$ and $K/K'$ is unramified. Then the lemma immediately follows because those $K'$ are pairwise distinct as they ramified at different primes.
		
		Let $\varpi$ be the quotient map $(C_2)^{\oplus n} \to (C_2)^{\oplus d}$ that sends $I_i$ to $I'_i$ for each $i=1, \ldots, n$. Assume $K\in \calS((C_2)^{\oplus n}, n, \{I_i\}_{i=1}^n, 1)$ satisfies \eqref{eq:up-n} and let $\iota: \Gal(K/\Q) \to (C_2)^{\oplus n}$ be the isomorphism used in Definition~\ref{def:not}. Then $\varpi \circ \iota: \Gal(K/\Q) \twoheadrightarrow (C_2)^{\oplus d}$ defines a subfield $K'$ in $\calS((C_2)^{\oplus d}, n, \{I'_i\}_{i=1}^n, 1)$ such that $K/K'$ is unramified. We let $\varphi$ denote the composite map as follows
		\begin{equation}\label{eq:varphi}
			\varphi: \Conv_{\frakp \in \Ram(K/\Q)} \langle x_{\frakp} \mid x_{\frakp}^2 \rangle \overset{\phi}{\longrightarrow} \Gal(K_{S_{\infty}}(2)/\Q) \longrightarrow \Gal(K/\Q),
		\end{equation}
		where $\phi$ is from Corollary~\ref{cor:Q}. Then	 $K$ satisfying \eqref{eq:up-n} implies that $\ker \phi \subset \Phi(\ker\varphi)$. Because $K/K'$ is unramified, $K_{S_{\infty}}(2)/K'$ is unramified outside $S_{\infty}$. Then, since we have a surjection
		\[
			\varphi':=\varpi \circ \iota \circ \varphi: \Conv_{\frakp \in \Ram(K/\Q)} \langle x_{\frakp} \mid x_{\frakp}^2 \rangle \longrightarrow \Gal(K'/\Q),
		\]
		we have $\ker\phi \subset \Phi(\ker \varphi) \subset \Phi(\ker \varphi')$, which implies that $K'$ satisfies \eqref{eq:up-d}.
	\end{proof}
	
	For the rest of this subsection, we let $K$ be a field in $\calS((C_2)^{\oplus n}, n, \{I_i\}_{i=1}^n, 1)$, and $\varphi$ be as described in \eqref{eq:varphi}. Recall that $\varphi$ induces a short exact sequence 
	\[
		1 \longrightarrow N \longrightarrow F \longrightarrow \Gal(K/\Q) \simeq (C_2)^{\oplus n} \longrightarrow 1
	\]
	as defined in \eqref{eq:ses}, where $F= (\Conv_{\frakp \in \Ram(K/\Q)} \langle x_{\frakp} \mid x_{\frakp}^2 \rangle ) / \Phi(\ker \varphi)$. With a slight abuse of notation, we let $x_{\frakp}$ and $y_{\frakp}$ (defined in Corollary~\ref{cor:Q}) also denote their images in $F$. 
	We will define surjections $\pi_{U, \frakp}: F \to C_2 \wr(C_2)^{\oplus \# U}$, for $U \subsetneq \Ram(K/\Q)$ and $\frakp \in \Ram(K/\Q) \backslash U$, that are compatible with different choices of $U$. To do that, we first introduce our notation for wreath product. Noting that
	\[
		C_2 \wr (C_2)^{\oplus \# U} \cong \F_2[(C_2)^{ \oplus \# U}] \rtimes (C_2)^{\oplus \# U},
	\]
	we choose an element $a^U$ whose normal closure is $\F_2[(C_2)^{\oplus \# U}]$, and for each $\frakq \in U$, we let $g^U_{\frakq}$ denote the generator of the coordinate for $\frakq$ in the subgroup $(C_2)^{\oplus \# U}$ of the wreath product (when we label the coordinates of $(C_2)^{\oplus \#U}$ by primes in $U$).

	\begin{lemma}\label{lem:red}
		For $\frakq \in U \subsetneq \Ram(K/\Q)$, consider a surjective homomorphism
		\begin{eqnarray*}
			f_{U, \frakq}: C_2 \wr (C_2)^{\oplus \# U} &\longrightarrow& C_2 \wr (C_2)^{\oplus \#(U \backslash \{ \frakq\})} \quad \text{defined by} \\
			a^U &\longmapsto& a^{U\backslash \{\frakq\}} \\
			g^U_{\frakp} &\longmapsto& \begin{cases}
				g^{U\backslash\{\frakq\}}_{\frakp} & \text{if } \frakq \neq \frakp \in U \\
				1 & \text{if } \frakq=\frakp.
			\end{cases}
		\end{eqnarray*}
		Then, for an element $y$ of the normal subgroup $\F_2[(C_2)^{\oplus \# U}]$ of $C_2 \wr (C_2)^{\oplus \# U}$, 
		\[
			[g_{\frakq}^U, y]=1 \quad \Longleftrightarrow \quad f_{U, \frakq}(y)=1.
		\]
	\end{lemma}
	
	\begin{proof}
		The module structure of the subgroup $\F_2[(C_2)^{\oplus \# U}]$ is defined by the conjugation in $C_2 \wr (C_2)^{\oplus \# U}$. So an element $y$ in this subgroup satisfies $[g_{\frakq}^U, y]=1$ if and only if $y$ belongs to the submodule $M$ generated by $a^U+g_{\frakq}^U(a^U)$ (note that $M$ is the maximal submodule on which the action of $g_{\frakq}^U$ is trivial). Note that $M$ is a subgroup of $\F_2[(C_2)^{\oplus \# U}]$ of $\F_2$-dimension $|C_2^{\#U-1}|$. Then, since $f_{U, \frakq}(M)=1$ and $f_{U, \frakq}(\F_2[(C_2)^{\oplus \# U}])=\F_2[(C_2)^{ \oplus \# (U\backslash \{\frakq\})}]$, it follows that $\ker f_{U, \frakq} \cap \F_2[(C_2)^{\oplus \# U}]=M$. So the equivalence in the lemma follows.
	\end{proof}

	Recall that $x_{\frakq}$, $\frakq \in \Ram(K/\Q)$ form a generating set of $F$. For each choice of $U$ and $\frakp$, we define
	\begin{eqnarray*}
		\pi_{U, \frakp}: F &\longrightarrow& C_2 \wr (C_2)^{\oplus \# U} \\
		x_{\frakq} &\longmapsto& \begin{cases}
			g^U_{\frakq} & \text{if } \frakq \in U \\
			a_U &\text{if }\frakq = \frakp \\
			1 &\text{otherwise.}
		\end{cases}
	\end{eqnarray*}
	One can check that $f_{U, \frakq} \circ \pi_{U, \frakp} = \pi_{U\backslash \{\frakq\}, \frakp}$ for any $\frakq \in U$. 
	
	\begin{lemma}\label{lem:int-pi}
		The homomorphism $\pi_{U, \frakp}$ is surjective, and
		\[
			\bigcap_{U, \frakp} \ \ker \pi_{U, \frakp} =1
		\]
		where the intersection runs over all $U \subsetneq \Ram(K/\Q)$ and $\frakp \in \Ram(K/\Q) \backslash U$.
	\end{lemma}
	
	\begin{proof}
		By the definition of $F$ and the structure of the wreath product, $\ker \pi_{U, \frakp}$ is the normal subgroup of $F$ generated by
		\begin{eqnarray*}
			x_{\frakq} && \text{if $\frakq \not\in U \cup \{ \frakp \}$, and}\\
			{}[x_{\frakq_1}, x_{\frakq_2}] && \text{if $\frakq_1, \frakq_2 \in U$}.
		\end{eqnarray*}
		So one can check by Lemma~\ref{lem:basis}\eqref{item:basis-2} that if an element of $F$ is contained in $\ker \pi_{U, \frakp}$ for all $U$ and $\frakp$ then it has to be trivial.
	\end{proof}
	
	\begin{proposition}\label{prop:K-P}
		The followings are equivalent. 
		\begin{enumerate}
			\item \label{item:K-P-1} $\dim_{\F_2} \Cl^+(K)[2]=(n-2) 2^{n-1}+1$,
			\item \label{item:K-P-2} $y_{\frakq} \in \ker \pi_{U, \frakp}$, for any $U \subsetneq \Ram(K/\Q)$, $\frakp\neq \frakq \in \Ram(K/\Q) \backslash U$.
		\end{enumerate}
	\end{proposition}
	
	\begin{remark}
		This proposition reconstructs the equivalence demonstrated in \cite[Thm.4.2 and Thm.4.4]{Koymans-Pagano} in the case that, using their language, each coordinate of the acceptable vectors is a prime number. A more complicated construction of maps from $F$ to wreath products can reconstruct this equivalence for arbitrary acceptable vectors.
	\end{remark}
	
	\begin{proof}
				Let $H^+(K)$ denote the maximal multiquadratic subextension of $K$ inside $K_{S_{\infty}}(2)$. Consider the following composition of surjections
		\[
			\psi: \Conv_{\frakp \in \Ram(K/\Q)} \langle x_{\frakp} \mid x_{\frakp}^2 \rangle \overset{\phi}{\longrightarrow} \Gal(K_{\infty}(2)/\Q) \longrightarrow \Gal(H^+(K)/\Q),
		\]
		where $\phi$ is from Corollary~\ref{cor:Q} and the second map is the natural surjection of Galois groups. Let $\varphi$ be as defined in \eqref{eq:varphi}. Then these maps satisfy 
		\begin{equation}\label{eq:maps}
			\ker \psi =  \Phi (\ker \varphi) \ker \phi.
		\end{equation}
		
		Recall that, by Corollary~\ref{cor:Q-rank}, the statement \eqref{item:K-P-1} holds if and only if $\dim_{\F_2} \Cl^+(K)[2]$ reaches its maximum, which happens if and only if 
		\begin{equation}\label{eq:when-max}
			\ker \psi = \Phi (\ker \varphi).
		\end{equation}
		Since $K$ is assumed to be totally real, $x_{\infty} \in \ker \varphi$ and therefore $x_{\infty}^2 \in (\ker \varphi)^2$. 
		Then by the description of $\ker\phi$ in Corollary~\ref{cor:Q} and the identities \eqref{eq:maps} and \eqref{eq:when-max}, to prove the equivalence between \eqref{item:K-P-1} and \eqref{item:K-P-2}, it suffices to prove that \eqref{item:K-P-2} is equivalent to $[x_{\frakq}, y_{\frakq}]=1$ in $F$ for each  $\frakq \in \Ram(K/\Q)$,
		and hence by Lemma~\ref{lem:int-pi}, it suffices to prove that \eqref{item:K-P-2} is equivalent to 
		\begin{equation}\label{eq:star}
			\pi_{U, \frakp}([x_{\frakq}, y_{\frakq}])=1 \text{ for any $U, \frakp \in \Ram(K/\Q) \backslash U$ and any $\frakq \in \Ram(K/\Q)$}.
		\end{equation}
		
		{\bf \eqref{item:K-P-2} $\Longrightarrow$ \eqref{eq:star}:} When $\frakq \neq \frakp$ are not in $U$, then \eqref{item:K-P-2} implies $\pi_{U, \frakp}([x_{\frakq}, y_{\frakq}])=1$ because of $\pi_{U, \frakp}(x_{\frakq})=1$. 
		When $\frakq \in U$ (then automatically $\frakp \neq \frakq)$, \eqref{item:K-P-2} ensures that $\pi_{U \backslash \{\frakq \}, \frakp} (y_{\frakq})=1$, then by Lemma~\ref{lem:red} that $\pi_{U, \frakp}([x_{\frakq}, y_{\frakq}])=1$. 
		
		For the last case, i.e., when $\frakp=\frakq$ (automatically $\frakq \not\in U$), we consider the epimorphism
		\[
			\alpha_U: F \overset{\pi_{U, \frakp}}{\longrightarrow} C_2 \wr (C_2)^{\oplus \# U} \longrightarrow (C_2)^{\oplus \# U},
		\]
		where the second arrow is the natural projection. By definition, $\alpha_U$ does not depend on the choice of $\frakp$. Because \eqref{item:K-P-2} implies $\alpha_U(y_{\frakq})=1$. Since $\frakq \not\in U$, both $\pi_{U, \frakp}(x_{\frakq})$ and $\pi_{U, \frakp}(y_{\frakq})$ are contained in $\ker (C_2 \wr (C_2)^{\oplus \# U} \to (C_2)^{\oplus \# U})$ which is abelian,
		so we obtain $\pi_{U, \frakp}([x_{\frakq}, y_{\frakq}]) =1$. 
		
		{\bf \eqref{eq:star} $\Longrightarrow$ \eqref{item:K-P-2}:} We prove by contradiction, and suppose that there exist $U \subsetneq \Ram(K/\Q)$ and $\frakp \neq \frakq \in \Ram(K/\Q) \backslash U$ such that $y_{\frakq} \not\in \ker \pi_{U, \frakp}$. Then by Lemma~\ref{lem:red}, we obtain $\pi_{U \cup \{\frakq\}, \frakp}([x_{\frakq}, y_{\frakq}])\neq 1$, which contradicts to \eqref{eq:star}.
	\end{proof}



\begin{bibdiv}
\begin{biblist}

\bib{BBH-imaginary}{article}{
      author={Boston, Nigel},
      author={Bush, Michael~R.},
      author={Hajir, Farshid},
       title={Heuristics for {$p$}-class towers of imaginary quadratic fields},
        date={2017},
        ISSN={0025-5831},
     journal={Math. Ann.},
      volume={368},
      number={1-2},
       pages={633\ndash 669},
         url={https://doi-org.proxy.lib.umich.edu/10.1007/s00208-016-1449-3},
      review={\MR{3651585}},
}

\bib{BBH-real}{article}{
      author={Boston, Nigel},
      author={Bush, Michael~R.},
      author={Hajir, Farshid},
       title={Heuristics for {$p$}-class towers of real quadratic fields},
        date={2021},
        ISSN={1474-7480},
     journal={J. Inst. Math. Jussieu},
      volume={20},
      number={4},
       pages={1429\ndash 1452},
  url={https://doi-org.proxy2.library.illinois.edu/10.1017/S1474748019000641},
      review={\MR{4293801}},
}

\bib{Fouvry-Kluners-1}{incollection}{
      author={Fouvry, \'{E}tienne},
      author={Kl\"{u}ners, J\"{u}rgen},
       title={Cohen-{L}enstra heuristics of quadratic number fields},
        date={2006},
   booktitle={Algorithmic number theory},
      series={Lecture Notes in Comput. Sci.},
      volume={4076},
   publisher={Springer, Berlin},
       pages={40\ndash 55},
         url={https://doi-org.proxy2.library.illinois.edu/10.1007/11792086_4},
      review={\MR{2282914}},
}

\bib{Fouvry-Kluners-2}{article}{
      author={Fouvry, \'{E}tienne},
      author={Kl\"{u}ners, J\"{u}rgen},
       title={On the 4-rank of class groups of quadratic number fields},
        date={2007},
        ISSN={0020-9910},
     journal={Invent. Math.},
      volume={167},
      number={3},
       pages={455\ndash 513},
  url={https://doi-org.proxy2.library.illinois.edu/10.1007/s00222-006-0021-2},
      review={\MR{2276261}},
}

\bib{Fouvry-Koymans}{article}{
      author={Fouvry, \'{E}tienne},
      author={Koymans, Peter},
       title={On {Dirichlet} biquadratic fields},
        date={2020},
        note={preprint, arXiv:2001.05350},
}

\bib{Fouvry-Koymans-Pagano}{article}{
      author={Fouvry, \'{E}tienne},
      author={Koymans, Peter},
      author={Pagano, Carlo},
       title={O{N} {THE} {$4$}-{RANK} {OF} {CLASS} {GROUPS} {OF} {DIRICHLET}
  {BIQUADRATIC} {FIELDS}},
        date={2022},
        ISSN={1474-7480},
     journal={J. Inst. Math. Jussieu},
      volume={21},
      number={5},
       pages={1543\ndash 1570},
  url={https://doi-org.proxy2.library.illinois.edu/10.1017/S1474748020000651},
      review={\MR{4476122}},
}

\bib{Gauss}{book}{
      author={Gauss, Carl~Friedrich},
       title={Disquisitiones arithmeticae},
   publisher={Springer-Verlag, New York},
        date={1986},
        ISBN={0-387-96254-9},
        note={Translated and with a preface by Arthur A. Clarke, Revised by
  William C. Waterhouse, Cornelius Greither and A. W. Grootendorst and with a
  preface by Waterhouse},
      review={\MR{837656}},
}

\bib{Gerth84}{article}{
      author={Gerth, Frank, III},
       title={The {$4$}-class ranks of quadratic fields},
        date={1984},
        ISSN={0020-9910},
     journal={Invent. Math.},
      volume={77},
      number={3},
       pages={489\ndash 515},
         url={https://doi-org.proxy2.library.illinois.edu/10.1007/BF01388835},
      review={\MR{759260}},
}

\bib{Gerth86}{article}{
      author={Gerth, Frank, III},
       title={Densities for certain {$l$}-ranks in cyclic fields of degree
  {$l^n$}},
        date={1986},
        ISSN={0010-437X},
     journal={Compositio Math.},
      volume={60},
      number={3},
       pages={295\ndash 322},
  url={http://www.numdam.org.proxy2.library.illinois.edu/item?id=CM_1986__60_3_295_0},
      review={\MR{869105}},
}

\bib{HMR20}{article}{
      author={Hajir, Farshid},
      author={Maire, Christian},
      author={Ramakrishna, Ravi},
       title={Infinite class field towers of number fields of prime power
  discriminant},
        date={2020},
        ISSN={0001-8708},
     journal={Adv. Math.},
      volume={373},
       pages={107318, 8},
  url={https://doi-org.proxy2.library.illinois.edu/10.1016/j.aim.2020.107318},
      review={\MR{4130462}},
}

\bib{HMR21}{article}{
      author={Hajir, Farshid},
      author={Maire, Christian},
      author={Ramakrishna, Ravi},
       title={On the {S}hafarevich group of restricted ramification extensions
  of number fields in the tame case},
        date={2021},
        ISSN={0022-2518},
     journal={Indiana Univ. Math. J.},
      volume={70},
      number={6},
       pages={2693\ndash 2710},
  url={https://doi-org.proxy2.library.illinois.edu/10.1512/iumj.2021.70.8755},
      review={\MR{4359923}},
}

\bib{Iwasawa-solvable}{article}{
      author={Iwasawa, Kenkichi},
       title={On solvable extensions of algebraic number fields},
        date={1953},
        ISSN={0003-486X},
     journal={Ann. of Math (2)},
      volume={58},
       pages={548\ndash 572},
         url={https://doi-org.proxy2.library.illinois.edu/10.2307/1969754},
      review={\MR{0059314}},
}

\bib{Iwasawa}{article}{
      author={Iwasawa, Kenkichi},
       title={On {G}alois groups of local fields},
        date={1955},
        ISSN={0002-9947},
     journal={Trans. Amer. Math. Soc.},
      volume={80},
       pages={448\ndash 469},
         url={https://doi-org.proxy2.library.illinois.edu/10.2307/1992998},
      review={\MR{75239}},
}

\bib{Klys}{article}{
      author={Klys, Jack},
       title={The distribution of {$p$}-torsion in degree {$p$} cyclic fields},
        date={2020},
        ISSN={1937-0652},
     journal={Algebra Number Theory},
      volume={14},
      number={4},
       pages={815\ndash 854},
  url={https://doi-org.proxy2.library.illinois.edu/10.2140/ant.2020.14.815},
      review={\MR{4114057}},
}

\bib{Koch}{book}{
      author={Koch, Helmut},
       title={Galois theory of {$p$}-extensions},
      series={Springer Monographs in Mathematics},
   publisher={Springer-Verlag, Berlin},
        date={2002},
        ISBN={3-540-43629-4},
  url={https://doi-org.proxy2.library.illinois.edu/10.1007/978-3-662-04967-9},
        note={With a foreword by I. R. Shafarevich, Translated from the 1970
  German original by Franz Lemmermeyer, With a postscript by the author and
  Lemmermeyer},
      review={\MR{1930372}},
}

\bib{Koymans-Pagano-genus}{article}{
      author={Koymans, Peter},
      author={Pagano, Carlo},
       title={Higher genus theory},
        date={2022},
        ISSN={1073-7928},
     journal={Int. Math. Res. Not. IMRN},
      number={4},
       pages={2772\ndash 2823},
  url={https://doi-org.proxy2.library.illinois.edu/10.1093/imrn/rnaa196},
      review={\MR{4381932}},
}

\bib{Koymans-Pagano-Gerth}{article}{
      author={Koymans, Peter},
      author={Pagano, Carlo},
       title={On the distribution of {${\rm Cl}(K)[l^\infty]$} for degree {$l$}
  cyclic fields},
        date={2022},
        ISSN={1435-9855},
     journal={J. Eur. Math. Soc. (JEMS)},
      volume={24},
      number={4},
       pages={1189\ndash 1283},
         url={https://doi-org.proxy2.library.illinois.edu/10.4171/JEMS/1112},
      review={\MR{4397040}},
}

\bib{Koymans-Pagano}{article}{
      author={Koymans, Peter},
      author={Pagano, Carlo},
       title={A sharp upper bound for the 2-torsion of class groups of
  multiquadratic fields},
        date={2022},
        ISSN={0025-5793},
     journal={Mathematika},
      volume={68},
      number={1},
       pages={237\ndash 258},
         url={https://doi-org.proxy2.library.illinois.edu/10.1112/mtk.12123},
      review={\MR{4405977}},
}

\bib{Kluners-Wang}{article}{
      author={Kl\"{u}ners, J\"{u}rgen},
      author={Wang, Jiuya},
       title={{$\ell$}-torsion bounds for the class group of number fields with
  an {$\ell$}-group as {G}alois group},
        date={2022},
        ISSN={0002-9939},
     journal={Proc. Amer. Math. Soc.},
      volume={150},
      number={7},
       pages={2793\ndash 2805},
         url={https://doi-org.proxy2.library.illinois.edu/10.1090/proc/15882},
      review={\MR{4428868}},
}

\bib{Liu2020}{article}{
      author={Liu, Yuan},
       title={Presentations of {G}alois groups of maximal extensions with
  restricted ramification},
        date={2020},
        note={preprint, arXiv:2005.07329},
}

\bib{Liu-rou}{article}{
      author={Liu, Yuan},
       title={Non-abelian cohen--lenstra heuristics in the presence of roots of
  unity},
        date={2022},
        note={preprint, arXiv:2202.09471},
}

\bib{LWZB}{article}{
      author={Liu, Yuan},
      author={Wood, Melanie~Matchett},
      author={Zureick-Brown, David},
       title={A predicted distribution for {Galois} groups of maximal
  unramified extensions},
        date={2019},
        note={preprint, arXiv:1907.05002},
}

\bib{Maire96}{article}{
      author={Maire, Christian},
       title={Finitude de tours et {$p$}-tours {$T$}-ramifi\'{e}es
  mod\'{e}r\'{e}es, {$S$}-d\'{e}compos\'{e}es},
        date={1996},
        ISSN={1246-7405},
     journal={J. Th\'{e}or. Nombres Bordeaux},
      volume={8},
      number={1},
       pages={47\ndash 73},
         url={http://jtnb.cedram.org/item?id=JTNB_1996__8_1_47_0},
      review={\MR{1399946}},
}

\bib{NSW}{book}{
      author={Neukirch, J\"{u}rgen},
      author={Schmidt, Alexander},
      author={Wingberg, Kay},
       title={Cohomology of number fields},
     edition={Second},
      series={Grundlehren der mathematischen Wissenschaften [Fundamental
  Principles of Mathematical Sciences]},
   publisher={Springer-Verlag, Berlin},
        date={2008},
      volume={323},
        ISBN={978-3-540-37888-4},
         url={https://doi.org/10.1007/978-3-540-37889-1},
      review={\MR{2392026}},
}

\bib{Reichardt}{article}{
      author={Reichardt, Hans},
       title={Konstruktion von {Z}ahlk\"{o}rpern mit gegebener {G}aloisgruppe
  von {P}rimzahlpotenzordnung},
        date={1937},
        ISSN={0075-4102},
     journal={J. Reine Angew. Math.},
      volume={177},
       pages={1\ndash 5},
  url={https://doi-org.proxy2.library.illinois.edu/10.1515/crll.1937.177.1},
      review={\MR{1581540}},
}

\bib{Sha-190}{article}{
      author={\v{S}afarevi\v{c}, I.~R.},
       title={Construction of fields of algebraic numbers with given solvable
  {G}alois group},
        date={1954},
        ISSN={0373-2436},
     journal={Izv. Akad. Nauk SSSR. Ser. Mat.},
      volume={18},
       pages={525\ndash 578},
      review={\MR{0071469}},
}

\bib{Scholz}{article}{
      author={Scholz, Arnold},
       title={Konstruktion algebraischer {Z}ahlk\"{o}rper mit beliebiger
  {G}ruppe von {P}rimzahllpotenzordnung {I}},
        date={1937},
        ISSN={0025-5874},
     journal={Math. Z.},
      volume={42},
      number={1},
       pages={161\ndash 188},
         url={https://doi-org.proxy2.library.illinois.edu/10.1007/BF01160071},
      review={\MR{1545668}},
}

\bib{Smith}{article}{
      author={Smith, Alexander},
       title={{$2^{\infty}$-Selmer groups, $2^{\infty}$-class groups, and
  Goldfeld's conjecture}},
        date={2017},
        note={preprint, arXiv:1702.02325},
}

\bib{Smith2022}{article}{
      author={Smith, Alexander},
       title={{The distribution of $\ell^{\infty}$-Selmer groups in degree
  $\ell$ twist families}},
        date={2022},
        note={preprint, arXiv:2207.05674},
}

\end{biblist}
\end{bibdiv}
\end{document}